\documentclass[11pt,reqno]{amsart}
\usepackage[foot]{amsaddr}

\usepackage{amsmath,amsthm,bm}
\usepackage{graphicx}
\usepackage[labelfont=bf,font=footnotesize,indention=0.3cm,margin=0cm]{caption}  
\usepackage[font=footnotesize,margin=5pt,indention=5pt]{subfig}          
\usepackage{amssymb,bbm}
\usepackage[linktocpage,a4paper,
            colorlinks=true,linkcolor=black,
            citecolor=black,urlcolor=black,
            pdfpagemode=UseNone,
            bookmarks=true,
            bookmarksopen=false,
            hyperfootnotes=false,
            pdfhighlight=/I,
            pdftitle={Fast Simulation of Large-Scale Growth Models},
            pdfauthor={Tobias Friedrich and Lionel Levine}]{hyperref}
\usepackage{multirow}
\usepackage{booktabs}
\usepackage[usenames]{color}
\usepackage{tikz}
\usetikzlibrary{arrows,shapes,backgrounds}

\usepackage{setspace}   

\usepackage[numbers,sort&compress,longnamesfirst]{natbib}

\def\source{\mathtt{s}}
\def\target{\mathtt{t}}
\def\top{\mathtt{Top}}
\def\inrad{r_{\mathit{in}}}
\def\outrad{r_{\mathit{out}}}

\def\Re{\mathrm{Re}}

\newcommand{\floor}[1]{\left\lfloor {#1} \right\rfloor}

\newcommand{\old}[1]{}

\newtheorem{theorem}{Theorem}

\newtheorem{lemma}[theorem]{Lemma}
\newtheorem{corollary}[theorem]{Corollary}

\theoremstyle{remark}
\newtheorem*{remark}{Remark}

\theoremstyle{definition}
\newtheorem*{definition}{Definition}

\newcommand{\thmref}[1]{Theorem~\ref{thm:#1}}

\newcommand{\figref}[1]{Figure~\ref{fig:#1}}

\newcommand{\tabref}[1]{Table~\ref{tab:#1}}
\newcommand{\tabrefs}[2]{Tables~\ref{tab:#1} and~\ref{tab:#2}}
\newcommand{\Secref}[1]{\textsection\ref{sec:#1}}

\newcommand{\secref}[1]{\textsection\ref{sec:#1}}

\newcommand{\eq}[1]{equation~\eqref{eq:#1}}
\newcommand{\eqs}[2]{equations~\eqref{eq:#1} and~\eqref{eq:#2}}

\def\supp{\mbox{supp}}

\def\block{\mathop{block}}
\def\rnd{\mathop{rnd}}
\def\Binomial{\mbox{Binomial}}
\def\diff{\mbox{di{}f{}f}}
\def\avgdiff{\overline{\mbox{di{}f{}f}}}
\newcommand{\rotseq}[1]{\texttt{#1}}

\DeclareSymbolFont{AMSb}{U}{msb}{m}{n}
\DeclareMathSymbol{\C}{\mathbin}{AMSb}{"43}
\DeclareMathSymbol{\EE}{\mathbin}{AMSb}{"45}
\DeclareMathSymbol{\N}{\mathbin}{AMSb}{"4E}
\DeclareMathSymbol{\PP}{\mathbin}{AMSb}{"50}
\DeclareMathSymbol{\Q}{\mathbin}{AMSb}{"51}
\DeclareMathSymbol{\R}{\mathbin}{AMSb}{"52}
\DeclareMathSymbol{\Z}{\mathbin}{AMSb}{"5A}

\newcommand{\drawblob}[1]{
    \scalebox{.72}{
    \begin{tikzpicture}
        \node (full) at (0cm,0cm) [inner sep=0pt,above right]
            {\includegraphics[width=5cm]{#1}};
        \node (zoom) at (0.5cm,5.5cm) [thick,rectangle,draw,inner sep=2pt,above right]
            {\includegraphics[viewport=2.1in 4.6in 2.9in 4.98in,width=4cm,clip]{#1}};
        \draw[ultra thick, white] (2.1cm,4.6cm) rectangle (2.9cm,5cm);
        \draw[thick] (2.1cm,4.6cm) rectangle (2.9cm,5cm);
        \draw (2.1cm,4.98cm) -- (zoom.south west);
        \draw (2.9cm,4.98cm) -- (zoom.south east);
    \end{tikzpicture}
    }
}

\let\oldsqrt\sqrt

\def\DHLhksqrt#1#2{\setbox0=\hbox{$#1\oldsqrt{#2\,}$}\dimen0=\ht0
   \advance\dimen0-0.2\ht0
   \setbox2=\hbox{\vrule height\ht0 depth -\dimen0}%
   {\box0\lower0.4pt\box2}}
\renewcommand{\leq}{\leqslant}
\renewcommand{\geq}{\geqslant}
\renewcommand\epsilon\varepsilon

\newcount\minute \newcount\hour \newcount\hourMins
\def\now{\minute=\time \hour=\time \divide \hour by 60 \hourMins=\hour \multiply\hourMins by 60
  \advance\minute by -\hourMins \zeroPadTwo{\the\hour}:\zeroPadTwo{\the\minute}}

\def\today{\the\year-\zeroPadTwo{\the\month}-\zeroPadTwo{\the\day}}
\def\zeroPadTwo#1{\ifnum #1<10 0\fi #1}

\allowdisplaybreaks[1]   

\title[Fast Simulation of Large-Scale Growth Models]{Fast Simulation of Large-Scale Growth Models$^*$}
\thanks{$^*$ A conference version appeared in the
     15th International Workshop on Randomization and Computation (RANDOM 2011).
     The second author was partly supported by a National Science Foundation Postdoctoral Fellowship.}
     
\author{Tobias Friedrich$^1$}
\address{$^1$ Department 1: Algorithms and Complexity\\Max-Planck-Institut f\"ur Informatik\\66123 Saarbr\"ucken, Germany}

\author{Lionel Levine$^2$}
\address{$^2$ Department of Mathematics\\Massachusetts Institute of Technology\\Cambridge, MA 02139, USA}

\keywords{Cycle popping, internal diffusion limited aggregation, least action principle, low discrepancy random stack, odometer function, potential kernel, rotor-router model}
\subjclass[2010]{82C24, 05C81, 05C85}

\begin{document}

\maketitle

\begin{abstract}
We give an algorithm that computes the final state of certain growth models without computing all intermediate states.  Our technique is based on a ``least action principle" which characterizes the odometer function of the growth process.  Starting from an approximation for the odometer, we successively correct under- and overestimates and provably arrive at the correct final state.

Internal diffusion-limited aggregation (IDLA) is one of the models amenable to our technique.  The boundary fluctuations in IDLA were recently proved to be at most logarithmic in the size of the growth cluster, but the constant in front of the logarithm is still not known.  As an application of our method, we calculate the size of fluctuations over two orders of magnitude beyond previous simulations, and use the results to estimate this constant.
\end{abstract}


\section{Introduction}

In this paper we study the \emph{abelian stack model}, a type of growth process on graphs.
Special cases include \emph{internal diffusion limited aggregation} (IDLA) and \emph{rotor-router aggregation}.
We describe a method for computing the final state of the process, given an initial approximation.  The more accurate the approximation, the faster the computation.

\subsection*{IDLA}
Starting with~$N$ chips at the origin of the two-dimensional square grid $\Z^2$,
each chip in turn performs a simple random walk until reaching an unoccupied site.
Introduced by \citet{MD} and independently by \citet{DF}, IDLA models physical phenomena such as a
solid melting around a heat source,
electrochemical polishing,
and fluid flow in a Hele-Shaw cell.
\citet{LBG} showed that as $N \to \infty$, the asymptotic shape of the resulting cluster of~$N$ occupied sites is a disk (and in higher dimensions, a Euclidean ball).

The boundary of an IDLA cluster is a natural model of a random propagating front (\figref{big}, left).  From this perspective, the most basic question one could ask is, what is the scale of the fluctuations around the limiting circular shape?  Until recently this was a long-standing open problem in statistical physics. It is now known that the fluctuations in dimension $2$ are of order at most $\log N$ \cite{JLS1,AG2}; however, it is still an open problem to show that the fluctuations are at least this large.  We give numerical evidence that $\log N$ is in fact the correct order, and estimate the constant in front of the log.

\subsection*{Rotor-router aggregation}
James Propp~\citep{Kleber} proposed the following way of derandomizing IDLA.
At each lattice site in~$\Z^2$ is a \emph{rotor} that can point north, east, south or west.  Instead of stepping in a random direction, a chip rotates the rotor at its current location counterclockwise, and then steps in the direction of this rotor.  Each of~$N$ chips starting at the origin walks in this manner until reaching an unoccupied site.
Given the initial configuration of the rotors (which can be taken, for example, to be all north), the resulting growth process is entirely deterministic.
Regardless of the initial rotor configurations, the asymptotic shape is a disk (and in higher dimensions, a Euclidean ball) and the inner fluctuations are proved to be $O(\log N)$~\citep{LP09a}.  The true fluctuations appear to grow even more slowly, and may even be bounded independent of~$N$.

Rotor-router aggregation is remarkable in that it generates a nearly perfect disk in the square lattice without any reference to the Euclidean norm $(x^2+y^2)^{1/2}$. Perhaps even more remarkable are the patterns formed by the final directions of the rotors (\figref{big}, right).

\begin{figure}[ptb]
    \scalebox{1.6}{
    \begin{tikzpicture}
        \node (fullIDLA) at (2cm,0cm) [inner sep=0pt,above right]
            {\includegraphics[height=4cm]{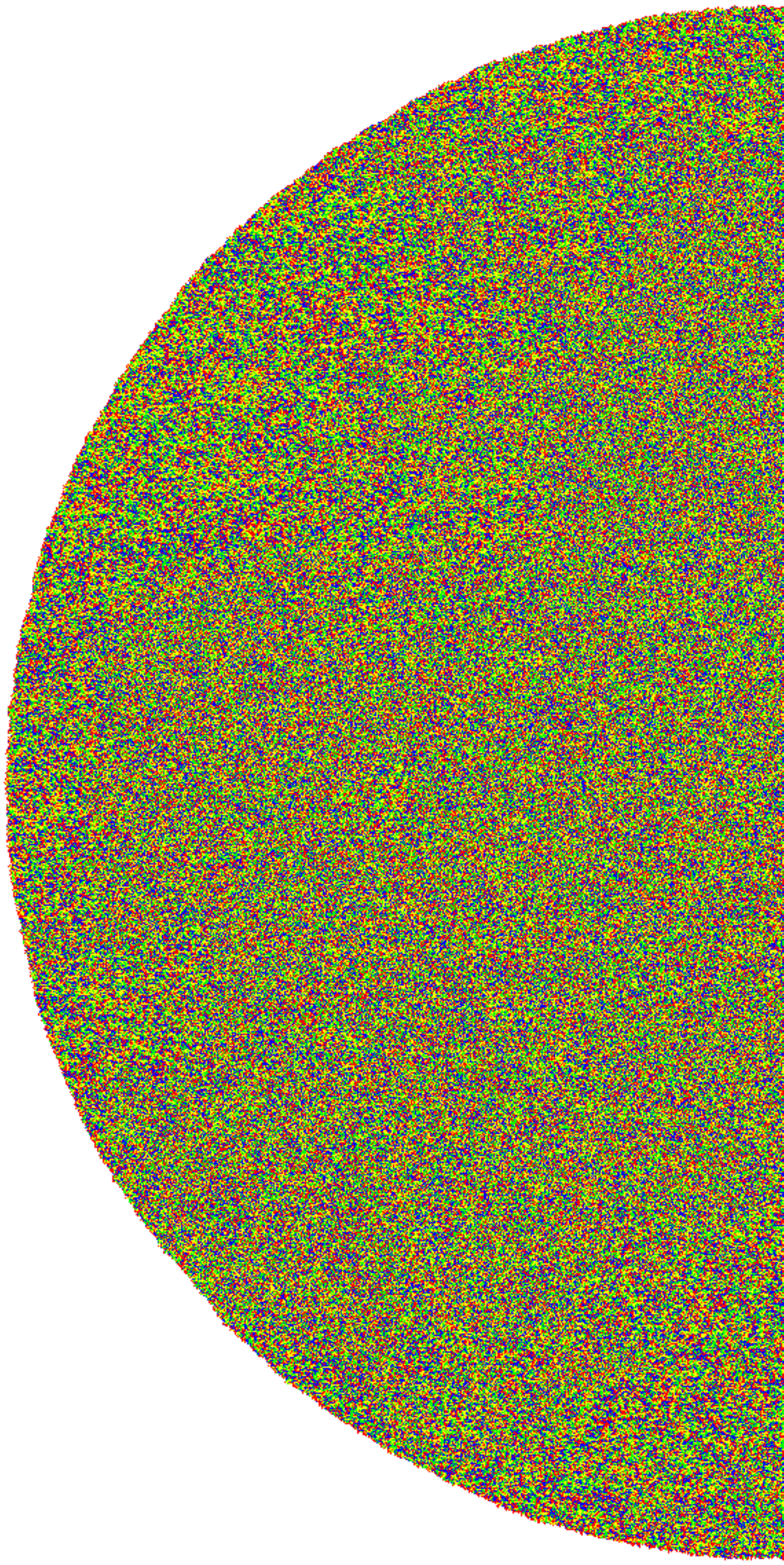}};
        \node (fullRR) at (3.9956cm,0cm) [inner sep=0pt,above right]
            {\includegraphics[height=4cm]{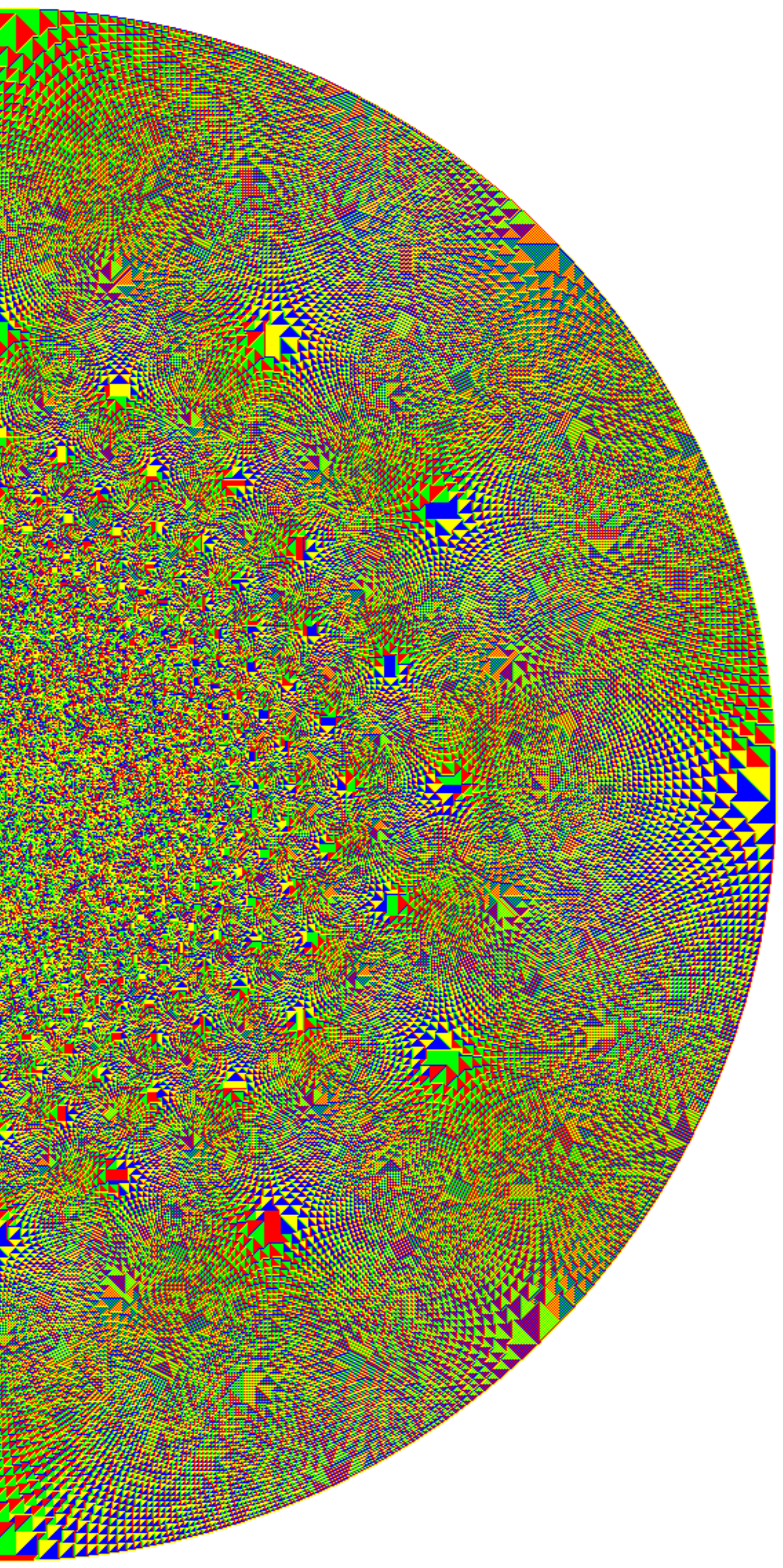}};
        \node (zoomIDLA) at (-.3cm,0.5cm) [thick,rectangle,draw,inner sep=1pt,above right]
            {\includegraphics[height=3cm]{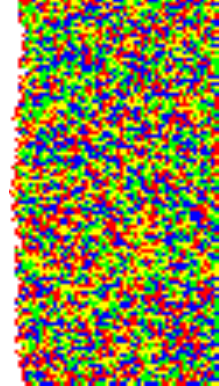}};
        \node (zoomRR) at (6.5cm,0.5cm) [thick,rectangle,draw,inner sep=1pt,above right]
            {\includegraphics[height=3cm]{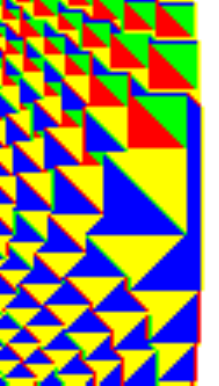}};

        \draw[thick,white] (5.765cm-0.0044cm,1.795cm) rectangle (6.01cm-0.0044cm+0.01cm,2.215cm);
        \draw[semithick] (5.765cm-0.0044cm,1.795cm) rectangle (6.01cm-0.0044cm+0.01cm,2.215cm);
        \draw (6.01cm-0.0044cm+0.01cm,2.215cm) -- (zoomRR.north west);
        \draw (6.01cm-0.0044cm+0.01cm,1.795cm) -- (zoomRR.south west);

        \draw[thick,white] (2.002cm-0.01cm,1.795cm) rectangle (2.247cm,2.215cm);
        \draw[semithick] (2.002cm-0.01cm,1.795cm) rectangle (2.247cm,2.215cm);
        \draw (2.002cm-0.01cm,2.215cm) -- (zoomIDLA.north east);
        \draw (2.002cm-0.01cm,1.795cm) -- (zoomIDLA.south east);
        \draw[thick,dashed] (4.003cm,-.05cm) -- (4.003cm,4.05cm);
    \end{tikzpicture}
    }

    \caption{IDLA cluster (left) and rotor-router cluster
        with counterclockwise rotor sequence (right) of $N=10^6$ chips.
        Half of each circular cluster is shown.
        Each site is colored according to the final direction of the rotor on top of its stack
            (yellow=\rotseq{W},
             red=\rotseq{S},
             blue=\rotseq{E},
             green=\rotseq{N}).
        Note that the boundary of the rotor-router cluster is much smoother than the boundary of the IDLA cluster.
        Larger rotor-router clusters of size up to $N=10^{10}$
        can be found at~\cite{hugerotor}.
             }
    \label{fig:big}
\end{figure}
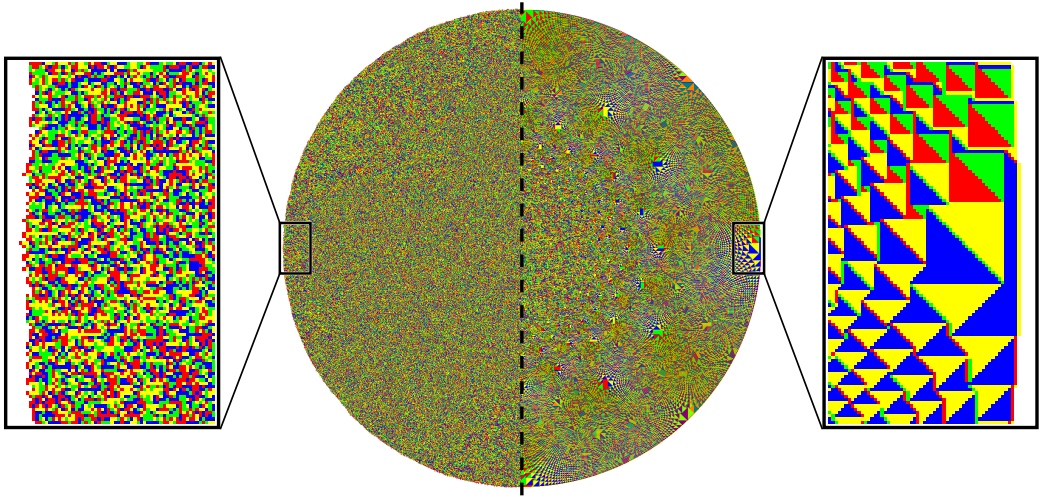

\subsection*{Low-discrepancy random stack}
To better understand whether it is the regularity or the determinism which
makes rotor-router aggregation so round, we follow a suggestion
of James Propp and simulate a third model, \emph{low-discrepancy random stack},
which combines the randomness of IDLA and the regularity of the rotor-router
model.

\subsection*{Computing the odometer function}
The central tool in our analysis of all three models is the \emph{odometer function}, which measures the number of chips emitted from each site.
  The odometer function determines the shape of the final occupied cluster via a nonlinear operator that we call the \emph{stack Laplacian}.  Our main technical contribution is that even for highly non-deterministic models such as IDLA, one can achieve
\emph{fast exact calculation via intermediate approximation}.
Approximating our three growth processes
by an idealized model called the divisible sandpile, we can use the known
asymptotic expansion of the potential kernel of random walk on~$\Z^2$ to obtain an initial approximation of the odometer function.
We present a method for carrying out subsequent local corrections to provably transform this approximation into the exact odometer function, and hence compute the shape of the occupied cluster.
Our runtime depends strongly on the accuracy of the initial approximation.

\subsection*{Applications}
Traditional step-by-step simulation of all aforementioned models in $\Z^2$ requires a runtime of order $N^2$
to compute the occupied cluster.  
Using our new algorithm, we are able to generate large clusters faster:
our observed runtimes are about $N \log N$ for the rotor-router model 
and about $N^{1.5}$ for IDLA. 
By generating many independent
IDLA clusters, we estimate the order of fluctuations from circularity over two orders
of magnitude beyond previous simulations.  Our data strongly support
the findings of~\cite{MM} that the order of the maximum fluctuation for IDLA in $\Z^2$ is
logarithmic in~$N$.  Two proofs of an upper bound $C \log N$ on the maximum fluctuation for IDLA in $\Z^2$ have recently been announced: see \cite{AG1,AG2} and \cite{JLS1}.  While the implied constant $C$ in these bounds is large, our simulations suggest that the maximum fluctuation is only about $0.528 \ln N$.

For rotor-router aggregation we
achieve four orders of magnitude beyond previous simulations, which has
enabled us to generate fine-scaled examples of
the intricate patterns that form in the rotors on the tops of the stacks
at the end of the aggregation process (\figref{big}, right).
These patterns remain poorly understood even on a heuristic level.
We have used our algorithm
to generate a four-color $10$-gigapixel image~\cite{hugerotor} of the final
rotors for $N=10^{10}$ chips.  This file is so large that we had to use a Google
maps overlay to allow the user to zoom and scroll through the image.  Indeed,
the degree of speedup in our method was so dramatic that memory, rather than
time, became the limiting factor.

\subsection*{Related Work}

Unlike in a random walk, in a rotor-router walk each vertex serves its neighbors in a fixed order.  The resulting walk, which is completely deterministic, nevertheless closely
resembles a random walk in several
respects~\citep{CooperSpencer,DoerrF09,CooperDST07,CooperDFS10,HP,EJC1}.
The rotor-router mechanism also leads to improvements in algorithmic applications.
Examples include 
autonomous agents patrolling a territory~\cite{WLB96}, 
external mergesort~\citep{BarveGV97},
broadcasting information in networks~\cite{DFS08,DFS09},
and iterative load-balancing~\cite{FGS10}.

Abelian stacks (defined in the next section) are a way of indexing the steps of a walk by location and time rather than by time alone.  This fruitful idea goes back at least to~\citet[\textsection 4]{DF}.  \citet{Wilson} (see also~\cite{PW}) used this stack-based view of random walk in his algorithm for sampling a random spanning tree of a directed graph.  The final cycle-popping phase of our algorithm is directly inspired by Wilson's algorithm.
Our serial algorithm for IDLA also draws on ideas from the parallel algorithm of \citet{MM}.

Abelian stacks are a special case of \emph{abelian networks} \cite{DharADP, Lev11}, also called ``abelian distributed processors.''   In this viewpoint, each vertex is a finite automaton, or ``processor.''
The chips are called ``messages.''  When a processor receives a message, it can change internal state and also send one or more messages to neighboring processors according to its current internal state.
We believe that it might be possible to extend our method to other types of abelian networks, such as the Bak-Tang-Wiesenfeld sandpile model~\cite{BTW}.  Indeed, the initial inspiration for our work was the ``least action principle'' for sandpiles described in~\cite{FLP}.


\subsection*{Organization of the paper}
After formally defining the abelian stack model in \secref{formal}, we describe the mathematics underlying our algorithm in \secref{leastaction}.  The main result of \secref{leastaction} is \thmref{odomcharacterization}, which uniquely characterizes the odometer function by a few simple properties.  In \secref{thealgorithm} we describe the algorithm itself, and use \thmref{odomcharacterization} to prove its correctness.  \Secref{approxodo} discusses how to find a good approximation function to use as input to the algorithm.  Finally, \secref{models} describes our implementation and experimental results.


\section{Formal Model}
\label{sec:formal}

The underlying graph for the abelian stack model can be any finite or infinite directed graph~$G=(V,E)$.  Each edge $e \in E$ is oriented from its source vertex $\source(e)$ to its target vertex $\target(e)$.  Self-loops (edges $e$ such that $\source(e)=\target(e)$) and multiple edges (distinct edges $e,e'$ such that $\source(e)=\source(e')$ and $\target(e)=\target(e')$) are permitted.  We assume that~$G$ is \emph{locally finite} --- each vertex is incident to finitely many edges --- and \emph{strongly connected}: for any two vertices $x,y \in V$ there are directed paths from $x$ to $y$ and from $y$ to $x$.
At each vertex~$x\in V$ is an infinite stack of \emph{rotors} $(\rho_n(x))_{n \geq 0}$.
Each rotor $\rho_n(x)$ is an edge of~$G$ emanating from~$x$, that is, $\source(\rho_n(x))=x$.  We say that rotor $\rho_0(x)$ is ``on top'' of the stack.

A finite number of indistinguishable chips are dispersed on the vertices of~$G$ according to some prescribed initial configuration.
For each vertex~$x$, the first chip to visit~$x$ is absorbed there and never moves again.  Each subsequent chip arriving at~$x$ first shifts the stack at $x$ so that the new stack is $(\rho_{n+1}(x))_{n \geq 0}$.
After shifting the stack, the chip moves from $x$ to the other endpoint~$y = \target(\rho_1(x))$ of the rotor now on top.  We call this two-step procedure (shifting the stack and moving a chip) \emph{firing} the site~$x$.  The effect of this rule is that the $n$th time a chip is emitted from~$x$, it travels along the edge $\rho_n(x)$.

We will generally assume that the stacks are \emph{infinitive}: for each edge $e$, infinitely many rotors $\rho_n(\source(e))$ are equal to $e$.  If~$G$ is infinite, or if the total number of chips is at most the number of vertices, then this condition ensures that firing eventually stops, and all chips are absorbed.

We are interested in the set
of \emph{occupied sites}, that is, sites that absorb a chip.  The \emph{abelian property}~\cite[Theorem~4.1]{DF} asserts that this set does not depend on the order in which vertices are fired.  This property plays a key role in our method; we discuss it further in \secref{leastaction}.

\medskip

If the rotors $\rho_n(x)$ are independent and identically distributed random edges $e$ such that $\source(e)=x$, then we obtain IDLA.  For instance, in the case $G=\Z^2$, we can take the rotors $\rho_n(x)$ to be independent with the uniform distribution on the set of $4$ edges joining $x$ to its nearest neighbors $x \pm \mathbf{e}_1, x \pm \mathbf{e}_2$.  The special case of IDLA in which all chips start at a fixed vertex~$o$ is more commonly described as follows.
Let $A_1=\{o\}$, and for $N \geq 2$ define a random set~$A_N$ of~$N$ vertices of~$G$ according to the recursive rule
	\begin{equation} \label{growthrule} A_{N+1} = A_{N} \cup \{x_N\} \end{equation}
where $x_N$ is the
endpoint of a random walk
started at~$o$ and stopped when it first visits a site not in~$A_N$.
These random walks describe one particular sequence in which the vertices can be fired, for the initial configuration of~$N$ chips at~$o$.  The first chip is absorbed at $o$, and subsequent chips are absorbed in turn at sites $x_1,\ldots,x_{N-1}$.  When firing stops, the set of occupied sites is~$A_N$.

A second interesting case is deterministic: the sequence $\rho_n(x)$ is periodic in~$n$, for every vertex~$x$.  For example, on~$\Z^2$, we could take the top rotor in each stack to point to the northward neighbor, the next to the eastward neighbor, and so on.
This choice yields the model of rotor-router aggregation defined by Propp~\cite{Kleber} and analyzed in~\cite{LP08,LP09a}.  It is described by the growth rule~\eqref{growthrule}, where $x_N$ is the endpoint of a rotor-router walk started at the origin and stopped on first exiting $A_N$.


\section{Least Action Principle}
\label{sec:leastaction}

A \emph{rotor configuration} on~$G$ is a function
\[
    r \colon V \rightarrow E
\]
such that $\source(r(v))=v$ for all $v\in V$.  A \emph{chip configuration} on~$G$ is a function
\[
    \sigma \colon V \rightarrow \Z
\]
with finite support.  Note we do not require $\sigma \geq 0$.  If $\sigma(x)=m>0$, we say there are $m$ chips at vertex $x$; if $\sigma(x)=-m<0$, we say there is a hole of depth $m$ at vertex $x$.

For an edge $e$ and a nonnegative integer $n$, let
\begin{equation}
    \label{eq:rotorfrequencies}
    R_\rho(e,n) = \# \{ 1\leq k \leq n \mid \rho_k(\source(e))=e \}
\end{equation}
be the number of times $e$ occurs among the first $n$ rotors in the stack at the vertex~$\source(e)$ (excluding the top rotor $\rho_0(\source(e))$).  When no ambiguity would result, we drop the subscript $\rho$.

Write $\N$ for the set of nonnegative integers.  Given a function $u\colon V\to \N$, we would like to describe the net effect on chips resulting from firing each vertex $x \in V$ a total of~$u(x)$ times.   In the course of these firings, each vertex $x$ emits $u(x)$ chips, and receives $R_\rho(e,u(\source(e)))$ chips along each incoming edge $e$ with $\target(e)=x$.  This motivates the following definition.

\begin{definition}
The \emph{stack Laplacian} of a function
	$ u \colon V \to \N $
is the function
\[
    \Delta_\rho u \colon V \to \Z
\]
given by
	\begin{equation}
	\label{eq:stacklaplacian}
    	\Delta_\rho u (x) = \sum_{\target(e)=x} R_\rho(e, u(\source(e))) - u(x).
	\end{equation}
The sum is over all edges~$e$ with target vertex $\target(e)=x$.  We use the notation~$\Delta_\rho$ to emphasize the dependence (via~$R_\rho$) on the rotor stacks~$(\rho_k(x))_{k \geq 0}$.
\end{definition}

Given an initial chip configuration~$\sigma_0$, the configuration~$\sigma$ resulting from performing $u(x)$ firings at each site $x \in V$ is given by
\begin{equation}
\label{eq:basicodomrelation}
    \sigma = \sigma_0 + \Delta_\rho u.
\end{equation}
The rotor configuration on the tops of the stacks after these firings is also easy to describe.  We denote this configuration by $\top_\rho(u)$, and it is given by
\[
    \top_\rho(u)(x) = \rho_{u(x)}(x).
\]
We also write $E^u \rho$ for the collection of shifted stacks:
	\[ (E^u \rho)_k(x) = \rho_{u(x)+k}(x). \]
The stack Laplacian is not a linear operator, but it satisfies the relation
	\begin{equation} \label{skew-linearity} \Delta_\rho(u+v) = \Delta_\rho u + \Delta_{E^u \rho} v. \end{equation}

\begin{figure}
  \begin{center}
    \begin{tikzpicture}[shorten >=1pt,auto,thick,
    every node/.style={text centered,font=\normalsize,anchor=center}]
    \node[left] at (0,3.5) {$\rho_0\colon$};
    \node[left] at (0,3) {$\rho_1\colon$};
    \node[left] at (0,2.5) {$\rho_2\colon$};
    \node[left] at (0,2) {$\rho_3\colon$};
    
    \node[left,red] at (0,1) {$u\colon$};
    \node[left] at (0,.5) {$\Delta_\rho u\colon$};
    \node[left] at (0,0) {$\top_\rho(u)\colon$};

    \node at (0.3,4.3) {$v_1$};
    \node at (1.0,4.3) {$v_2$};
    \node at (1.7,4.3) {$v_3$};
    \node at (2.4,4.3) {$v_4$};
    \node at (3.1,4.3) {$v_5$};
    
    \draw[very thick] (0,4) -- (3.4,4);
    \fill[black] (0.3,4) circle (.1);
    \fill[black] (1.0,4) circle (.1);
    \fill[black] (1.7,4) circle (.1);
    \fill[black] (2.4,4) circle (.1);
    \fill[black] (3.1,4) circle (.1);
    
    \draw[->, blue] (0.5,3.5) -- (0.1,3.5);
    \draw[->] (0.8,3.5) -> (1.2,3.5);
    \draw[->] (1.9,3.5) -> (1.5,3.5);
    \draw[->] (2.2,3.5) -> (2.6,3.5);
    \draw[->, blue] (3.3,3.5) -> (2.9,3.5);
    
    \draw[->] (0.5,3) -> (0.1,3);
    \draw[->] (0.8,3) -> (1.2,3);
    \draw[->] (1.5,3) -> (1.9,3);
    \draw[->, blue] (2.6,3) -> (2.2,3);
    \draw[->] (3.3,3) -> (2.9,3);
    
    \draw[->] (0.1,2.5) -> (0.5,2.5);
    \draw[->, blue] (1.2,2.5) -> (0.8,2.5);
    \draw[->, blue] (1.5,2.5) -> (1.9,2.5);
    \draw[->] (2.2,2.5) -> (2.6,2.5);
    \draw[->] (2.9,2.5) -> (3.3,2.5);
    
    \draw[->] (0.1,2) -> (0.5,2);
    \draw[->] (1.2,2) -> (0.8,2);
    \draw[->] (1.9,2) -> (1.5,2);
    \draw[->] (2.6,2) -> (2.2,2);
    \draw[->] (2.9,2) -> (3.3,2);

    \draw[red] (0,3.3) -- (3.4,3.3);
    \draw[red] (0,3.2) -- (0.65,3.2) -- (0.65,2.3) -- (2.05,2.3) -- (2.05,2.8) -- (2.75,2.8) -- (2.75,3.2) -- (3.4,3.2);
    
    \node[red] at (0.3,1) {$0$};
    \node[red] at (1.0,1) {$2$};
    \node[red] at (1.7,1) {$2$};
    \node[red] at (2.4,1) {$1$};
    \node[red] at (3.1,1) {$0$};

    \node at (0.3,.5) {$1$};
    \node at (1.0,.5) {$-2$};
    \node at (1.7,.5) {$0$};
    \node at (2.4,.5) {$1$};
    \node at (3.1,.5) {$0$};
    
    \draw[->,blue] (0.5,0) -> (0.1,0);
    \draw[->,blue] (1.2,0) -> (0.8,0);
    \draw[->,blue] (1.5,0) -> (1.9,0);
    \draw[->,blue] (2.6,0) -> (2.2,0);
    \draw[->,blue] (3.3,0) -> (2.9,0);
    
    \end{tikzpicture}
  \end{center}
  \caption{An example of rotor stacks $\rho$ and the stack Laplacian $\Delta_\rho u$.  Here, the underlying graph is a path of length 5. For instance, the
middle vertex $v_3$ has $u(v_3)=2$ and $\Delta_\rho u(v_3) = 1+1-2
=0$, since each of its neighbors $v_2$ and $v_4$ has one rotor between
the red lines pointing to $v_3$.}
\label{fig:example}
\end{figure}
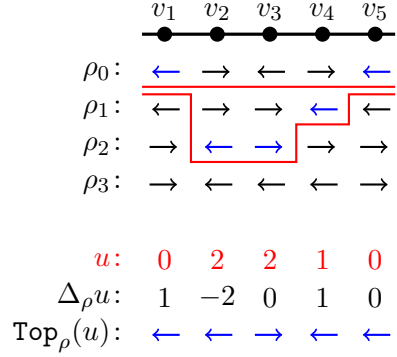

Vertices $x_1, \ldots, x_m$ form a \emph{legal firing sequence} for  $\sigma_0$ if
	\begin{equation*}
	\sigma_j(x_{j+1}) >1, \qquad j=0,\ldots,m-1
	\end{equation*}
where
	\[ \sigma_j = \sigma_0 + \Delta_\rho u_j \]
and
	\[ u_j(x) = \#\{i \leq j \colon x_i=x\}. \]
\figref{example} shows an example of rotor stacks and the stack Laplacian.

In words, the condition $\sigma_j(x_{j+1}) >1$ says that
after firing $x_1,\ldots,x_j$, the vertex $x_{j+1}$ has at least two chips.  We require at least two because in our growth model, the first chip to visit each vertex gets absorbed.

The firing sequence
is \emph{complete} if no further legal firings are possible; that is, $\sigma_m(x) \leq 1$ for all $x\in V$.  If $x_1,\ldots,x_m$ is a complete legal firing sequence for the chip configuration~$\sigma_0$, then we call the function $u:=u_m$ the \emph{odometer} of~$\sigma_0$.  The odometer tells us how many times each site fires.
\\ ~ \\
\textbf{Abelian Property} \cite[Theorem~4.1]{DF} Given an initial configuration~$\sigma_0$ and stacks~$\rho$, every complete legal firing sequence for $\sigma_0$ has the same odometer function~$u$.
\\ ~ \\
It follows that the final chip configuration $\sigma_m = \sigma_0 + \Delta_\rho u$ and the final rotor configuration~$\top_\rho(u)$ do not depend on the choice of complete legal firing sequence.

\begin{remark}
To ensure that $u$ is well-defined (i.e., that there exists a finite complete legal firing sequence) it is common to place some minimal assumptions on $\rho$ and $\sigma_0$.  For example, if $G$ is infinite and strongly connected, then it suffices to assume that the stacks $\rho$ are infinitive.
\end{remark}

Given a chip configuration~$\sigma_0$ and rotor stacks $(\rho_k(x))_{k \geq 0}$,
our goal is to compute  the final chip configuration~$\sigma_m$ without performing individual firings one at a time.  A fundamental observation is that by \eq{basicodomrelation}, it suffices to compute the odometer function~$u$ of~$\sigma_0$.  Indeed, once we know that each site $x$ fires $u(x)$ times, we can add up the number of chips $x$ receives from each of its neighbors and subtract the $u(x)$ chips it emits to figure out the final number of chips at $x$.  This arithmetic is accomplished by the term $\Delta_\rho u$ in \eq{basicodomrelation}; see Figure~\ref{fig:example} for an example.  In practice, it is usually easy to compute~$\Delta_\rho u$ given~$u$, an issue we address in \secref{thealgorithm}.

Our approach will be to start from an approximation of~$u$ and correct errors.  In order to know when our algorithm is finished, the key mathematical point is to find a list of properties of~$u$ that characterize it uniquely.  Our main result in this section, \thmref{odomcharacterization}, gives such a list.  As we now explain, the hypotheses of this theorem can all be guessed from certain necessary features of the final chip configuration~$\sigma_m$ and the final rotor configuration~$\top_\rho(u) $.  What is perhaps surprising is that these few properties suffice to characterize~$u$.

Let $x_1,\ldots,x_m$ be a complete legal firing sequence for the chip configuration~$\sigma_0$.  We start with the observation that since no further legal firings are possible, \begin{itemize}
\item $\sigma_m(x) \leq 1$ for all $x\in V$.
\end{itemize}
Next, consider the set $A$ of sites that fire, which is the support of $u$:
	\[ A = \supp(u) := \{x \in V \colon u(x)>0\}. \]
Since each site that fires must first absorb a chip, we have
\begin{itemize}
\item $\sigma_m(x)=1$ for all $x \in A$.
\end{itemize}
Finally, observe that for any vertex $x\in A$, the rotor $r(x) = \top_\rho(u)(x)$ at the top of the stack at~$x$ is the edge traversed by the last chip fired from $x$.  The last chip fired from a given finite subset $A'$ of $A$ must be to a vertex outside of $A'$, so $A'$ must have a vertex whose top rotor points outside of $A'$.
\begin{itemize}
\item For any finite set $A' \subset A$, there exists $x\in A'$ with $\target(r(x)) \notin A'$.
\end{itemize}
We can state this last condition more succinctly by saying that the rotor configuration $r=\top_\rho(u)$ is \emph{acyclic} on~$A$; that is, the spanning subgraph $(V,r(A))$ has no directed cycles.  Here $r(A) = \{r(x) \mid x \in A\}$.

\begin{theorem}
\label{thm:odomcharacterization}
Let $G$ be a finite or infinite directed graph, $\rho$ a collection of rotor stacks on $G$, and $\sigma_0$ a chip configuration on~$G$.  Fix $u_* \colon V \to \N$, and let $A_* = \mathrm{supp}(u_*)$.  Let $\sigma_* = \sigma_0 + \Delta_\rho u_*$, and suppose that
	\begin{itemize}
	\item $\sigma_* \leq 1$;
	\item $A_*$ is finite;
	\item $\sigma_*(x)=1$ for all $x \in A_*$; and
	\item $\top_\rho(u_*)$ is acyclic on $A_*$.
	\end{itemize}
Then there exists a finite complete legal firing sequence for $\sigma_0$, and its odometer function is $u_*$.
\end{theorem}

A useful mnemonic for Theorem~\ref{thm:odomcharacterization} is ``no hills, no holes, no cycles.''  A \emph{hill} is a site $x$ with $\sigma_*(x)>1$, and a \emph{hole} is a site $x$ with $\sigma_*(x) < 1_{x \in A_*}$.  Hills are forbidden everywhere ($\sigma_*(x)  \leq 1$ for all $x$), but it suffices to forbid holes and cycles only on $A_*$.

We break the proof of Theorem~\ref{thm:odomcharacterization} into two inequalities.  The first inequality can be seen as an analogue for the abelian stack model of the least action principle for sandpiles~\cite[Lemma~2.3]{FLP}.  

\begin{lemma}
\label{LAP}
(Least Action Principle)
If $\sigma_* \leq 1$ and $A_*$ is finite, then there exists a finite complete legal firing sequence for $\sigma_0$; and $u_* \geq u$, where $u$ is the odometer function of $\sigma_0$.
\end{lemma}

\begin{proof}
Perform legal firings in any order, without allowing any site~$x$ to fire more than $u_*(x)$ times, until no such firing is possible.  Since $A_*$ is finite, this procedure involves only finitely many firings.  Write $u'(x)$ for the number of times $x$ fires during this procedure.  We will show that this procedure gives a complete legal firing sequence, so that $u'=u$.

Write $\sigma' = \sigma_0 + \Delta_\rho u'$.  If $\sigma' \leq 1$, then $u' = u$ by the abelian property.  Otherwise, choose~$y$ such that $\sigma'(y)>1$.  We must have $u'(y) = u_*(y)$, or else it would have been possible to add another legal firing to $u'$.  Therefore, if we now perform $u_* - u'$ further firings, then since~$y$ does not fire, the number of chips at~$y$ cannot decrease.
Hence
	\[ \sigma_*(y) \geq \sigma'(y) > 1 \]
contradicting the assumption that $\sigma_* \leq 1$.
\end{proof}


\begin{lemma}
Suppose that
\begin{itemize}
	\item $A_*$ is finite;
	\item $\sigma_*(x) \geq 1$ for all $x \in A_*$; and
	\item $\top_\rho(u_*)$ is acyclic on $A_*$.
\end{itemize}
Then $u_* \leq u$.
\end{lemma}


\begin{proof}
Let
	\[ m(x) =\min(u(x),u_*(x)) \]
	\[ \psi = \sigma_0 + \Delta_\rho m \]
	\[ \sigma = \sigma_0 + \Delta_\rho u. \]
Then letting $\tilde{\rho} = E^m \rho$, we have from (\ref{skew-linearity})
	\begin{align*} \sigma  &= \sigma_0 + \Delta_\rho m + \Delta_{\tilde{\rho}} (u-m) \\
					     &= \psi + \Delta_{\tilde{\rho}} (u-m).
	\end{align*}
Likewise,
	$ \sigma_* = \psi + \Delta_{\tilde{\rho}} (u_*-m).$  Let
	\[ A = \{x \in V \mid u_*(x) > u(x)\}. \]
Since $u\geq 0$, we have $A \subset A_*$, hence $A$ is finite. We must show that $A$ is empty.

We have $\sigma_*(x) \geq 1$ for all $x \in A$ by hypothesis, while $\sigma(x) \leq 1$ by the definition of the odometer function~$u$.  So
	\begin{align*} 0 &\leq \sum_{x \in A} (\sigma_*(x) - \sigma(x)) \\
				&\leq \sum_{x \in A} \left(\Delta_{\tilde{\rho}}(u_*-m)(x) - \Delta_{\tilde{\rho}}(u-m)(x) \right).
	\end{align*}
For $x \in A$ we have $u(x) = m(x)$, so
	$\Delta_{\tilde{\rho}}(u-m)(x) \geq 0.$
Hence
	\begin{align*} 0 &\leq \sum_{x \in A} \Delta_{\tilde{\rho}}(u_*-m) \\
				&= \sum_{x \in A} \bigg( -(u_*(x)-m(x)) + \sum_{\target(e)=x} \#\{m(\source(e)) < k \leq u_*(\source(e)) \mid \rho_k(\source(e))=e \} \bigg).
	\end{align*}
The terms of the inner sum corresponding to edges $e$ such that $\source(e) \notin A$ vanish, since in that case $m(\source(e))=u_*(\source(e))$.  Hence
	\begin{align} \sum_{x \in A} (u_*(x)-m(x)) &\leq  \sum_{x\in A} \sum_{\substack{\target(e)=x \\ \source(e)\in A}} \#\{m(\source(e)) < k \leq u_*(\source(e)) \mid \rho_k(\source(e))=e \} \nonumber \\
		&=  \sum_{x \in A} \sum_{y \in A} \#\{m(y) < k \leq u_*(y) \mid \target(\rho_k(y))=x \} \nonumber \\
		&=  \sum_{y \in A} \#\{m(y) < k \leq u_*(y) \mid \target(\rho_k(y)) \in A\}.
		\label{eq:almostacontradiction}
	\end{align}
Now suppose for a contradiction that $A$ is nonempty.  Since $\top_\rho(u_*)$ is acyclic on~$A$, there exists a site $z\in A$ with $\target(\rho_k(z)) \notin A$, where $k=u_*(z)$.  Therefore the sum on the right side of (\ref{eq:almostacontradiction}) is strictly less than $\sum_{y \in A} (u_*(y)-m(y))$, which gives the desired contradiction.
\end{proof}

We conclude this section by observing a few consequences of \thmref{odomcharacterization}.  While our algorithm does not directly use the results below, we anticipate that they may be useful in further attempts to understand IDLA and rotor-router aggregation.

The stacks $\rho$ and initial configuration $\sigma_0$ determine an odometer function $u = u(\rho,\sigma_0)$, which is the unique function satisfying the hypotheses of \thmref{odomcharacterization}.  In particular, given~$\sigma_0$, the function~$u$ is completely characterized by properties of the chip configuration $\sigma_0 + \Delta_\rho u$ and the rotor configuration~$\top_\rho u$.  Since permuting the stack elements $\rho_1(x),\ldots,\rho_{u(\rho,\sigma_0)(x)-1}(x)$ does not change $\Delta_\rho u$ or $\top_\rho u$, we obtain the following result.

\begin{corollary}
\label{thm:exchangeability}
(Exchangeability) Let $\sigma$ be a chip configuration on $G$.  Let $(\rho_k(x))_{x \in V, k \in \N}$ and $(\rho'_k(x))_{x \in V, k \in \N}$ be two collections of rotor stacks, with the property that for each vertex $x \in V$, the rotors
	\[ \rho'_1(x), \ldots, \rho'_{u(\rho,\sigma)(x)-1}(x) \]
are a permutation of
	\[ \rho_1(x), \ldots, \rho_{u(\rho,\sigma)(x)-1}(x). \]
Suppose moreover that
	\[ \rho_{u(\rho,\sigma)(x)}(x) = \rho'_{u(\rho,\sigma)(x)}(x). \]
Then $u(\rho',\sigma) = u(\rho,\sigma)$.
\end{corollary}

\old{
\begin{proof}
Let $f = u(\rho,\sigma)$.  Then $\top_\rho f = \top_{\rho'} f$.  Moreover, $R_\rho(e,f(x)) = R_{\rho'}(e,f(x))$ for all edges $e$, hence
	\[ \Delta_\rho f = \Delta_{\rho'} f. \]
Since \thmref{odomcharacterization} characterizes $u(\rho,\sigma)$ in terms of properties only of $\Delta_\rho u$ and $\top_\rho(u)$, the proof is complete.
\end{proof}
}

Edges $e_1,\ldots,e_m \in E$ form a \emph{directed cycle} if $\source(e_{i+1})=\target(e_i)$ for $i=1,\dots,m-1$ and $\source(e_1)=\target(e_m)$.  The next result allows us to remove directed cycles of rotors from the stacks, without changing the final chip or rotor configuration.

\begin{corollary}
(Cycle removal)
Let  $(\rho_k(x))_{x \in V, k \in \N}$ be a collection of rotor stacks on~$G$, and let $(x_i,k_i)$ for $i=1,\ldots,m$ be distinct pairs such that the edges $\{\rho_{k_i}(x_i)\}_{i=1}^m$ form a directed cycle in $G$.  Let $\sigma$ be a chip configuration on $G$, and suppose that
$k_i \leq u(\rho,\sigma)(x_i)-1$ for all $i=1,\ldots,m$.  Let $\rho'$ be the rotor stacks obtained from $\rho$ by removing the rotors $\rho_{k_i}(x_i)$ for all $i=1,\ldots,m$ and re-indexing the remaining rotors in each stack by $\N$.  Then
	\[ u(\rho,\sigma) = u(\rho',\sigma) + \chi \]
where $\chi(x) = \# \{1\leq i \leq m \mid x_i=x\}$.  Moreover, the final chip and rotor configurations agree:
	\[ \sigma + \Delta_\rho [u(\rho,\sigma)] = \sigma + \Delta_{\rho'} [u(\rho',\sigma)] \]
	\[ \top_\rho [u(\rho,\sigma)] = \top_{\rho'}[u(\rho',\sigma)]. \]
\end{corollary}

\begin{proof}
Let $f=u(\rho,\sigma)$.  The bound on $k_i$ implies that $\top_\rho f = \top_{\rho'}(f-\chi)$.  By Theorem~\ref{thm:odomcharacterization}, to complete the proof it suffices to check that $\Delta_\rho f = \Delta_{\rho'}(f-\chi)$.  For any vertex $x$ and edge $e$ with $\source(e)=x$, we have
	\begin{align*} R_\rho(e,f(x)) &= \# \{ 1 \leq k \leq f(x) \mid \rho_k(x)=e \} \\
					&= R_{\rho'}(e,f(x)-\chi(x)) + c(e)
	\end{align*}
where $c(e) = \# \{1 \leq i \leq m \mid \rho_{k_i}(x_i)=e\}$.  Here we have used the fact that the pairs $(x_i,k_i)$ are distinct.  Hence
	\begin{align} \Delta_\rho f (x) &= -f(x) + \sum_{\target(e)=x} R_\rho(e,f(x)) \nonumber \\
						&= -f(x) + \sum_{\target(e)=x} R_{\rho'}(e, f(x)-\chi(x)) + \sum_{\target(e)=x} c(e). \label{eq:deltarho'off-chi}
	\end{align}
Since the edges $\{\rho_{k_i}(x_i)\}_{i=1}^m$ form a directed cycle, we have $\sum_{\target(e)=x} c(e) =  \sum_{\source(e)=x} c(e) = \chi(x)$.  So \eqref{eq:deltarho'off-chi} simplifies to $\Delta_{\rho'}(f-\chi)(x)$, which shows that $\Delta_\rho f = \Delta_{\rho'}(f-\chi)$.
\end{proof}


\section{The Algorithm: From Approximation to Exact Calculation}
\label{sec:thealgorithm}

\begin{figure}[ptb]
    \centering
    \subfloat[After odometer approximation ($u_1$)]{
        \drawblob{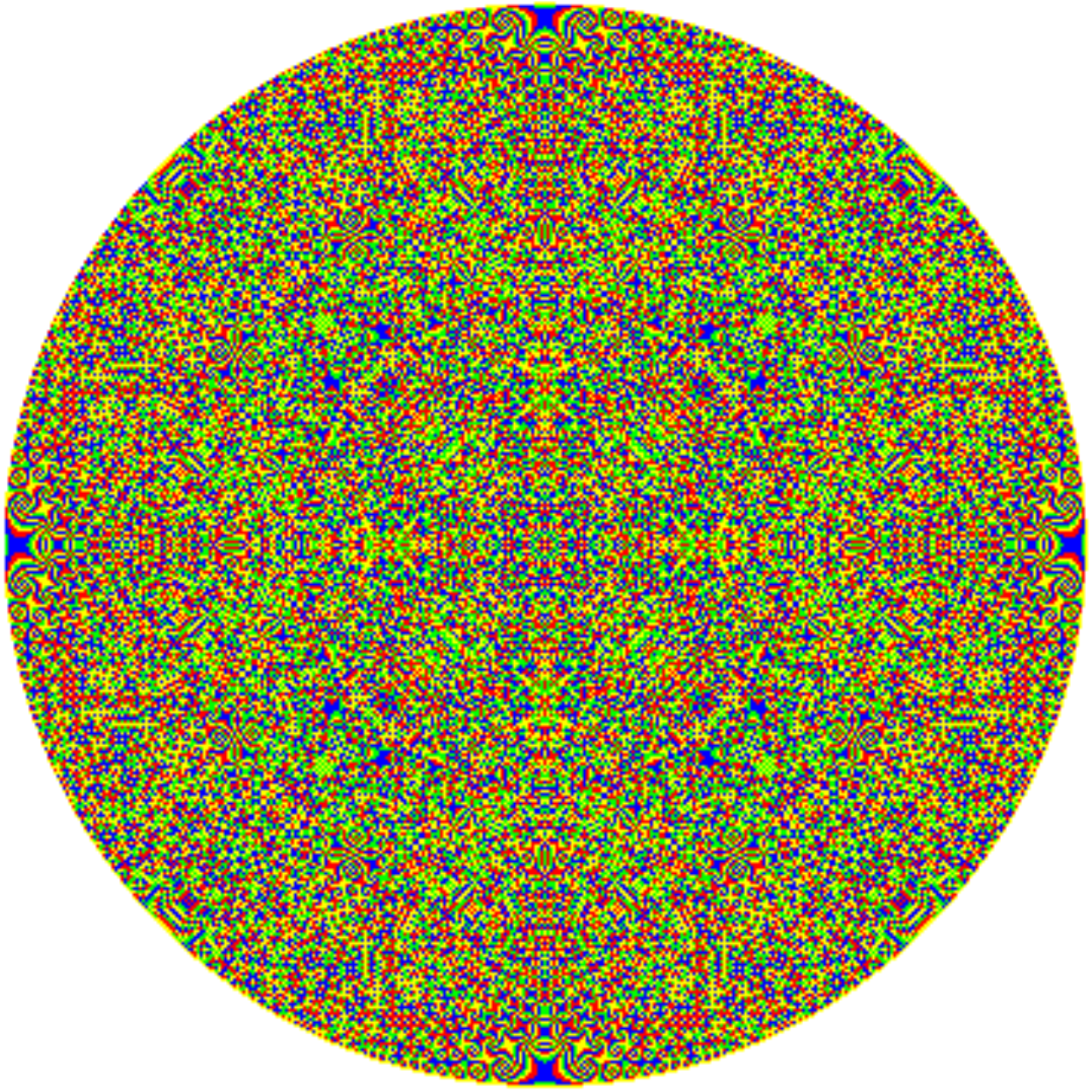}
    }
    \subfloat[After annihilation ($u_2$)]{
        \drawblob{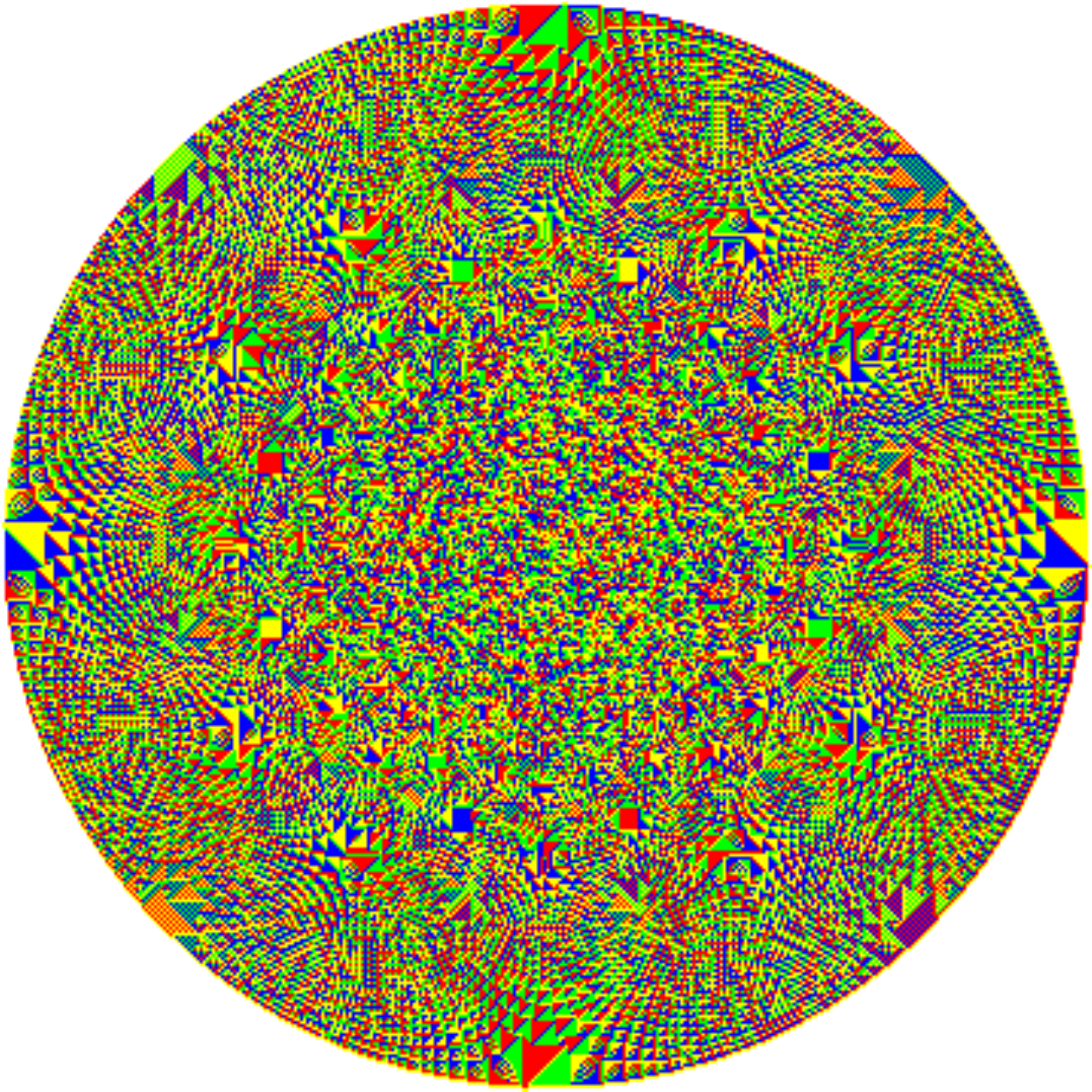}
    }
    \subfloat[After cycle popping ($u_3$)]{
        \drawblob{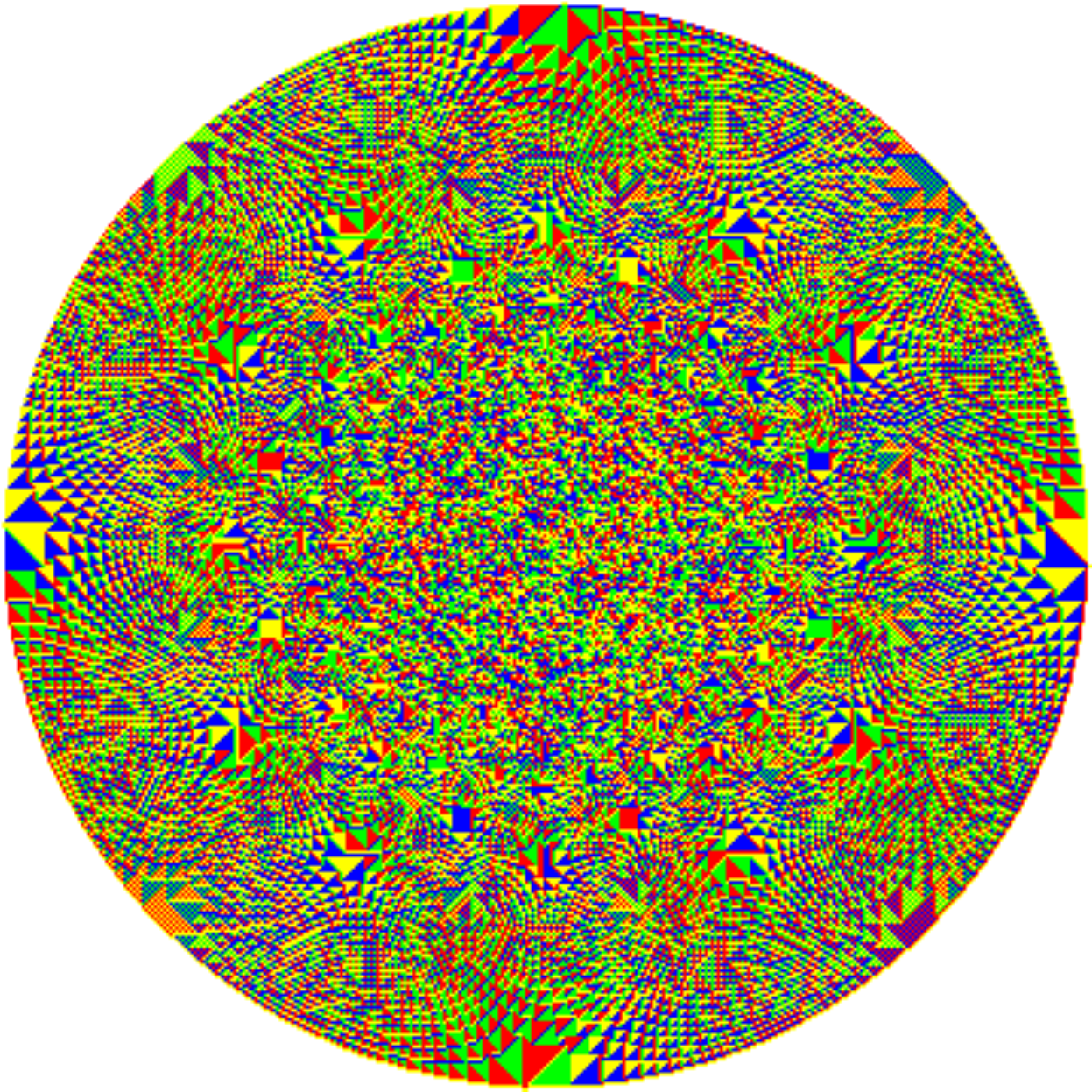}
    }
    \caption{Classic rotor router aggregation of $N=100,000$ chips with counterclockwise rotor sequence.
        The pictures show the direction of the rotors on top of the stacks
        after each step of the computation
            (yellow=\rotseq{W},
             red=\rotseq{S},
             blue=\rotseq{E},
             green=\rotseq{N}).}
    \label{fig:rotor}
\end{figure}

In this section we describe how to compute the odometer function~$u$ exactly, given as input an approximation~$u_1$.  The running time depends on the accuracy of the approximation, but the correctness of the output does not.  In the next section we explain how to find a good approximation~$u_1$ for the example of $N$ chips started at the origin in~$\Z^2$.

Recall that $G$ may be finite or infinite, and we assume that $G$ is strongly connected.  We assume that the initial configuration~$\sigma_0$ satisfies $\sigma_0(x) \geq 0$ for all $x$, and $\sum_x \sigma_0(x) < \infty$.  If $G$ is finite, we assume that $\sum_x \sigma_0(x)$ is at most the number of vertices of~$G$ (otherwise, some chips would never get absorbed).  The only assumption on the approximation~$u_1$ is that it is nonnegative with finite support.  Finally, we assume that the rotor stacks are infinitive, which ensures that the growth process terminates after finitely many firings: that is, $\sum_{x \in V} u(x) < \infty$.

For $x \in V$, write
	\[ d_{out}(x) = \# \{e \in E \mid \source(e)=x\} \]
	\[ d_{in}(x) = \# \{e \in E \mid \target(e)=x\} \]
for the out-degree and in-degree of~$x$.

The odometer function~$u$ depends on the initial chip configuration~$\sigma_0$ and on the rotor stacks~$(\rho_k(x))_{k \geq 0}$.  The latter are completely specified by the function~$R(e,n)$ defined in \secref{leastaction}.  Note that for rotor-router aggregation, since the stacks are periodic, $R(e,n)$ has the simple explicit form
\begin{equation}
    \label{eq:RRR}
    R(e,n) = \floor{\frac{n+d_{out}(x)-j}{d_{out}(x)}}
\end{equation}
where $j$ is the least positive integer such that $\rho_j(x)=e$.
For IDLA, $R(e,n)$ is a random variable with the $\Binomial(n,p)$ distribution, where $p$ is the transition probability associated to the edge~$e$.

In this section we take $R(e,n)$ as known.  From a computational standpoint, if the stacks are random, then determining $R(e,n)$ involves calls to a pseudorandom number generator.  We address the issue of minimizing the number of such calls in~\secref{exp:IDLA}.


Our algorithm consists of an approximation step followed by two error-correction steps: an annihilation step that corrects the chip locations, and a reverse cycle-popping step that corrects the rotors.

\begin{figure}[ptb]
    \centering
    \subfloat[After odometer approximation ($\sigma_1$)]{
        \drawblob{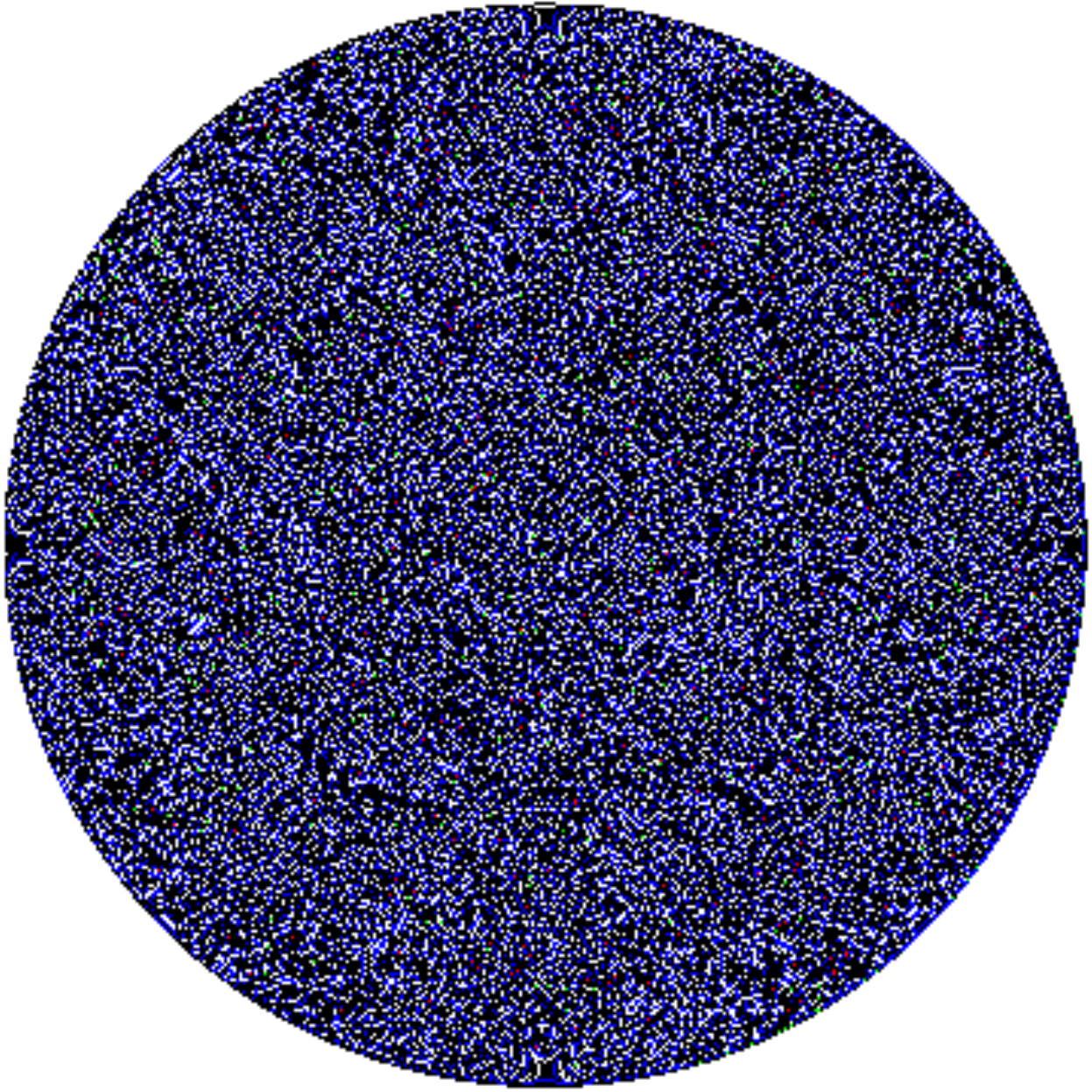}
    }
    \subfloat[After annihilation ($\sigma_2$)]{
        \drawblob{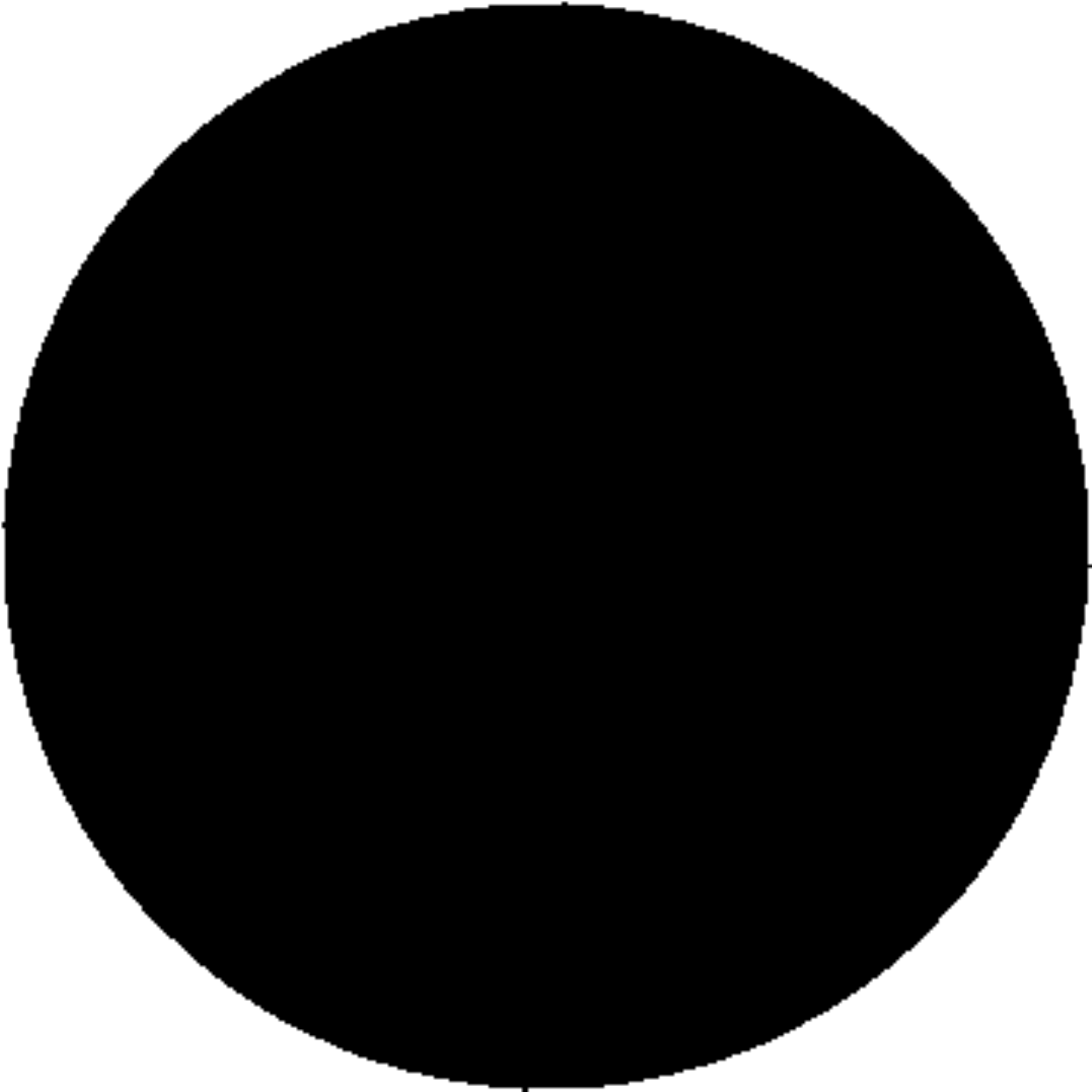}
    }
    \subfloat[After cycle popping ($\sigma_3$)]{
        \drawblob{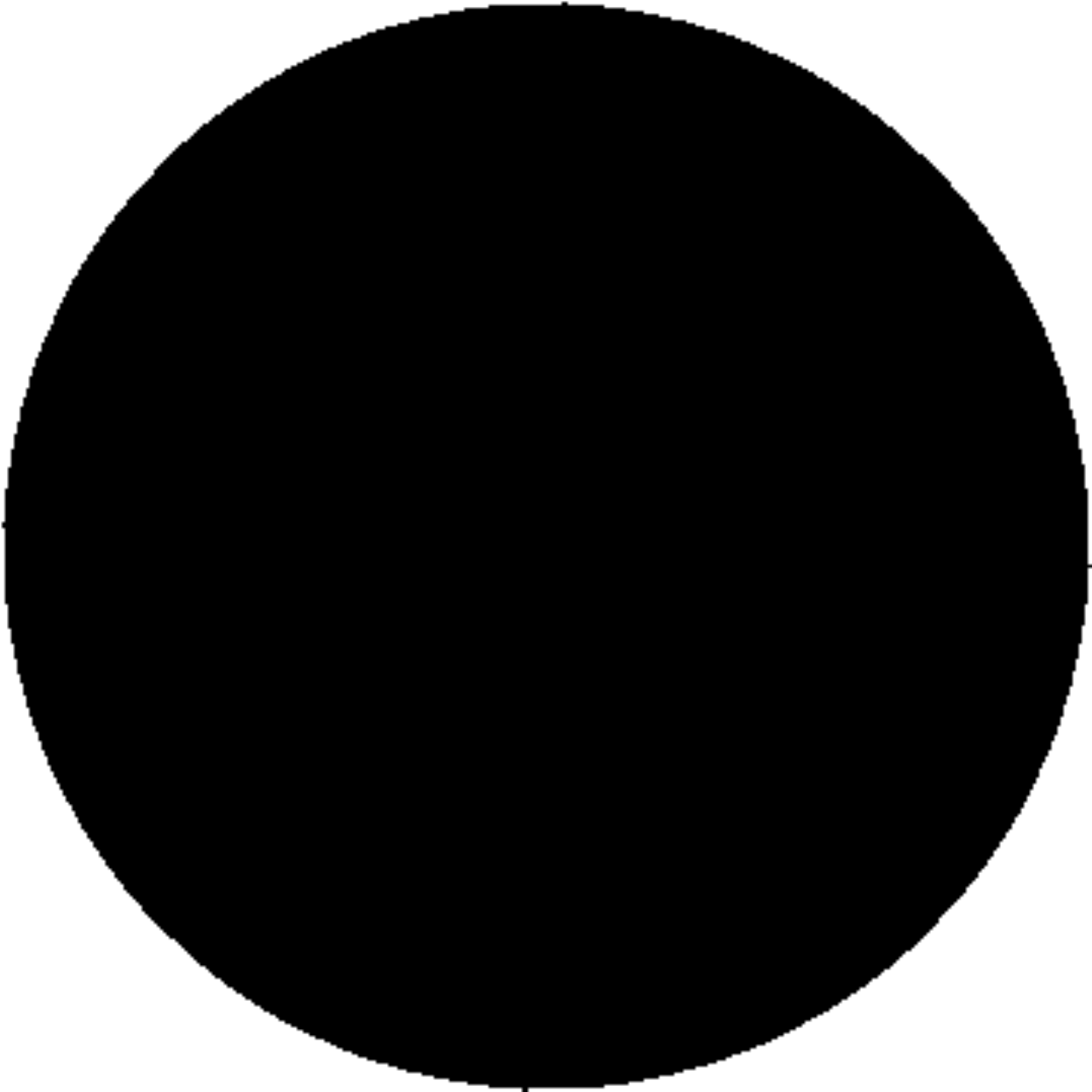}
    }
    \caption{Classic rotor router aggregation of $N=100,000$ chips with counterclockwise rotor sequence.
        The pictures show the number of chips after each step of the algorithm.
            (Location $x$ is colored
            red if $\sigma'(x)=-1$,
            white if $\sigma'(x)=0$,
            black if $\sigma'(x)=1$,
            blue if $\sigma'(x)=2$,
            green if $\sigma'(x)=3$.)
            Note that there are no locations with $\sigma'(x)<-1$ or $\sigma'(x)>3$, and that no chips move during the final cycle-popping phase.}
    \label{fig:chip}
\end{figure}

\begin{enumerate}
\item \textbf{Approximation.} Perform firings according to the approximate odometer, by computing the chip configuration $\sigma_1 = \sigma_0 + \Delta_\rho u_1$.
Using equation~(\ref{eq:stacklaplacian}), this takes time $O(d_{in}(x)+1)$ for each vertex $x$, for a total time of $O(\#E + \#V)$.  This step is where the speedup occurs, because we are performing many firings at once: $\sum_x u_1(x)$ is typically much larger than $\# E + \#V$.  Return~$\sigma_1$. \\

\item \textbf{Annihilation.} Start with $u_2=u_1$ and $\sigma_2=\sigma_1$.
If $x \in V$ satisfies $\sigma_2(x)>1$, then we call $x$ a \emph{hill}.
If $\sigma_2(x)<0$, or if $\sigma_2(x)=0$ and $u_2(x)>0$, then we call $x$ a \emph{hole}.
For each $x \in \Z^2$,
	\begin{enumerate}
	\item If $x$ is a hill, fire it by incrementing $u_2(x)$ by $1$ and then
	      moving one chip from $x$ to $\target(\top(u_2)(x))$.
	\item If $x$ is a hole, unfire it by moving one chip from
	      $\target(\top(u_2)(x))$ to $x$ and then decrementing $u_2(x)$ by one.	
	\end{enumerate}
	A hill can disappear in one of two ways: by reaching an unoccupied site on the boundary, or by reaching a hole and canceling it out.
    When there are no more hills and holes, return $u_2$. \\

\item \textbf{Reverse cycle-popping.}  Start with $u_3 = u_2$ and
	\[ A_3 = \{x \in V \colon u_3(x)>0\}. \]
If $\top(u_3)$ is not acyclic on~$A_3$, then pick a cycle and unfire each of its vertices once.  This may create additional cycles.   Update $A_3$ (it may shrink, since $u_3$ has decreased) and repeat until $\top(u_3)$ is acyclic on~$A_3$.
Output $u_3$.
\end{enumerate}

\noindent
Next we argue that the algorithm terminates, and that its final output~$u_3$ equals the odometer function~$u$.  Step~2 is simplest to analyze if we first fire all hills, and only after there are no more hills begin unfiring holes.  In practice, however, we found that it is much faster to fire hills and unfire holes in tandem; see \textsection\ref{sec:exp:implementation} for the details of our implementation.

At the beginning of step~2, all hills are contained in the set
	\[ S = \{x\in V \colon \sigma_1(x) > 0\}. \]
Since $\sigma_0$ and $u_1$ have finite support, $\sigma_1 = \sigma_0 + \Delta_\rho u_1$ has finite support, so~$S$ is finite.  Since the total number of chips is conserved, we have
	\[ \sum_{x \in V} \sigma_1(x) = \sum_{x \in V} \sigma_0(x). \]
The right side is $\leq \#V$ by assumption.  Therefore if $S=V$, we must have $\sigma_1(x) = 1$ for all $x\in V$; in this case there are no hills or holes, and we move on to step~3.

Suppose now that $S$ is a proper subset of~$V$.  Let
	\[ h = \sum_{x \in S} (\sigma_1(x)-1) \]
be the total height of the hills.  Note that firing a hill cannot increase~$h$.  If a given vertex fires infinitely often, then since the rotor stacks are infinitive, each of its out-neighbors also fires infinitely often; since~$G$ is strongly connected, it would follow that every vertex fires infinitely often.  Thus after firing finitely many hills, a chip must leave~$S$.  When this happens,~$h$ decreases.  Thus after finitely many firings we reach~$h=0$ and there are no more hills.

Next we begin unfiring the holes.  After all hills have been settled, we have $u_2(x) \geq 0$ for all $x\in V$.  The sum $\sum_{x \in V} u_2(x)$ is finite, and each unfiring decreases it by one.  To show that the unfiring step terminates, it suffices to show that for all $x \in V$ the unfiring of holes never causes $u_2(x)$ to become negative.  Indeed, suppose that $u_2(x)=0$ and $u_2(y) \geq 0$ for all neighbors $y$ of~$x$.  Then the number of chips at $x$ is $\sigma_0(x) + \Delta_\rho u_2(x) \geq 0$,
so~$x$ is not a hole.  Therefore the unfiring step terminates and its output~$u_2$ is nonnegative.

%
After step~2 there are no hills or holes, i.e., $0 \leq \sigma_2(x) \leq 1$ for all $x$, and if $\sigma_2(x)=0$ then $u_2(x)=0$.

During step~3, we unfire sites only within $A_3$.  Since $\sum_{x \in V} u_3(x)$ is finite and decreases with each unfiring, this step terminates and its output $u_3$ is nonnegative.
When a cycle is unfired, each vertex in the cycle sends a chip to the previous vertex, so there is no net movement of chips: $\sigma_3 = \sigma_2$.  In particular, there are no hills at the end of step~3.  If $\sigma_3(x)=0$, then $\sigma_2(x)=0$; since there were no holes at the end of step~2, this means that $u_2(x)=0$, and hence $u_3(x)=0$.  So there are still no holes at the end of step~3.  By construction, $\top(u_3)$ is acyclic on~$A_3$.  Therefore all conditions of \thmref{odomcharacterization} are satisfied, which shows that $u_3=u$ as desired.


\section{Approximating the Odometer Function}
\label{sec:approxodo}

Next we describe how to find a good approximation to the odometer to use as input to the algorithm described in \secref{thealgorithm}.  Our main assumption will be that the rotor stacks are \emph{balanced} in the sense that
\[
    R(e,n) \approx R(e',n)
\]
for all $n\in \N$ and all edges $e,e'$ with $\source(e)=\source(e')$.  By definition, rotor-router aggregation obeys the strong balance condition
\[
    |R(e,n) - R(e',n)| \leq 1.
\]
IDLA is somewhat less balanced: $|R(e,n)-R(e',n)|$ is typically on the order of~$\sqrt{n}$.  It turns out that this level of balance is still enough to get a fairly good approximation and hence a significant speedup in our algorithm.

If the rotor stacks are balanced, then the stack Laplacian $\Delta_\rho$ is well-approximated by the operator $\Delta$ on functions $u\colon V \to \Z$ defined by
\[
    \Delta u (z) = \sum_{\target(e)=z} \frac{u(\source(e))}{d_{out}(\source(e))} - u(z).
\]

In this setting we can approximate the behavior of our stack-based aggregation with
an idealized model called the \emph{divisible sandpile}~\cite{LP09a}.  Instead of discrete chips, each vertex $z$ has a real-valued ``mass'' $\sigma_0(z)$.  Any site with mass greater than $1$ can fire by keeping mass $1$ for itself, and distributing the excess mass to its out-neighbors by sending an equal amount of mass along each outgoing edge.  The resulting odometer function
	\[ v(z) = \mbox{total mass emitted from $z$} \]
satisfies the discrete variational problem

\begin{align}
\label{variational}
v &\geq 0 \nonumber \\
\Delta v &\leq 1 - \sigma_0 \\
v (\Delta v - 1 + \sigma_0) &= 0. \nonumber
\end{align}

In words, these conditions say that each site emits a nonnegative amount of mass, each site ends with mass at most~$1$, and each site that emits a positive amount of mass ends with mass exactly~$1$.
The conditions \eqref{variational} can be reformulated as an obstacle problem, that of finding the smallest superharmonic function lying above a given function; see~\cite{LP09b}.  That formulation shows existence and uniqueness of the solution~$v$.

If the rotor stacks are sufficiently balanced, we expect the divisible sandpile odometer function~$v$ to approximate closely our abelian stack odometer~$u$.  The next question is how to compute or approximate~$v$.  The obstacle problem formulation shows that~$v$ can be computed exactly by linear programming.
Such an approach works well for small to moderate system sizes, but for the sizes we are interested in, the number of variables~$v(z)$ is prohibitively large.

Fortunately, for specific examples it is sometimes possible to guess a near solution~$w \approx v$.  We briefly indicate how to do this for the specific example of interest to us, the initial configuration
	\[ \sigma_0 = N\delta_o \]
consisting of~$N$ chips at the origin $o \in \Z^2$.  In that case, the set
of sites that are fully occupied in the final divisible sandpile configuration $\sigma_0+\Delta v$ is very close to the disk
\[
    B_r = \{z \in \Z^2 \colon |z| < r \}
\]
of radius $r=\sqrt{N/\pi}$; see \cite[Theorem~3.3]{LP09a}.
Here $|z| = (z_1^2+z_2^2)^{1/2}$ is the Euclidean norm.
Thus we are seeking a function $w\colon \Z^2 \to \R$ satisfying
\begin{align*}
    \Delta w &= 1-N\delta_o &&\mbox{ in } B_r \\
	       w &\approx 0 &&\mbox{ on } \partial B_r.
\end{align*}
An example of such a function is
	\begin{equation}
	\label{quadraticplusgreen}
    	w(z) = |z|^2 - N a(z) - r^2 + N a((r,0))
	\end{equation}
where $a(z)$ is the potential kernel
for simple random walk $(X_n)_{n\geq 0}$ started at the origin in~$\Z^2$, defined as
\[
    a(z) = \sum_{n = 1}^\infty \left( \PP(X_n = o) - \PP(X_n=z) \right).
\]
Its discrete Laplacian is $\Delta a = \delta_o$.

As input to our algorithm we will use the function \[ w(z)^+ := \max(0,w(z)) \]
where $w(z)$ is given by~(\ref{quadraticplusgreen}).
  One computational issue remains, which is how to compute the potential kernel $a(z)$.
The potential kernel has the asymptotic expansion~\cite[Remark~2]{FU}
\begin{equation}
\label{eq:expansion}
    a(z) = \frac{2}{\pi} \ln |z| + \kappa +
        \frac{1}{6\pi} \frac{8\omega_1^2 \omega_2^2 - 1}{|z|^2} + O(|z|^{-4})
\end{equation}
where $\omega = z/|z|$ and $\kappa = \frac{\ln 8 + 2\gamma}{\pi}$; here $\gamma \approx 0.577216$ is Euler's constant $\lim (\sum_{k=1}^n \frac{1}{k} - \ln n)$.
Note that if $\theta$ is the argument of~$z$, then
\begin{align*}
    8 \omega_1^2 \omega_2^2 - 1 &=  8 \sin^2 \theta \cos^2 \theta - 1 \\
	&= 2 \sin^2 2\theta - 1 \\
	&= \sin^2 2\theta - \cos^2 2\theta \\
	&= -\cos 4\theta.
\end{align*}
Thus, identifying $\Z^2$ with $\Z + i\Z \subset \C$, we can write
	\[ a(z) = \frac{2}{\pi} \ln |z| + \kappa -
        \frac{1}{6\pi} \frac{\mbox{Re} (z^4) }{|z|^6} + O(|z|^{-4}). \]

For $z$ close to the origin the error term $O(|z|^{-4})$ becomes significant.  Therefore, we use the McCrea-Whipple algorithm~\citep{MW} (see also~\citep{KS}) to determine $a(z)$ exactly for $|z|<100$.  This algorithm uses the exact identity
	\[ a(n+in) = \frac{4}{\pi} \sum_{k=1}^n \frac{1}{2k-1} \]
for $n \geq 0$, together with the relation $\Delta a = \delta_o$ and reflection symmetry across the real and imaginary axes to compute $a(z)$ recursively.  The values of $a(z)$ for $z \in \Z+i\Z$ are rational linear combinations of $1$ and $\frac{1}{\pi}$.

Now we can describe the function~$u_1$ that we used as input to the first step of our algorithm.  Let $r = \sqrt{N/\pi}$.  Approximating the term~$a((r,0))$ in~(\ref{quadraticplusgreen}) by~$\frac{2}{\pi} \log r + \kappa$, we set
 	\[ u_1(z) = \left\lfloor |z|^2 + r^2\,(2 \ln r - 1 + \pi \kappa - \pi a(z)) \right\rceil, \qquad |z|<100. \]
Here $\lfloor t \rceil = \floor{t+\frac12}$ denotes the closest integer to $t\in \R$.
For $|z| \geq 100$ we use the asymptotic expansion for $a(z)$ in (\ref{quadraticplusgreen}), which gives
\begin{equation}
 u_1(z) = \left \lfloor |z|^2 + r^2 \left( 2 \ln \frac{r}{|z|} - 1 + \frac{\mbox{Re}(z^4)} { 6\,|z|^6}\right) \right \rceil^+, \qquad |z| \geq 100,
 \label{eq:odoapprox}
\end{equation}
where $t^+ := \max(t,0)$.
Including more terms of the asymptotic expansion of~$a(z)$ from~\cite{KS} improves the approximation very slightly, but increases the overall runtime.

\begin{figure}[tbp]
    \centering
    \subfloat[$N=100000$.]{
        \drawblob{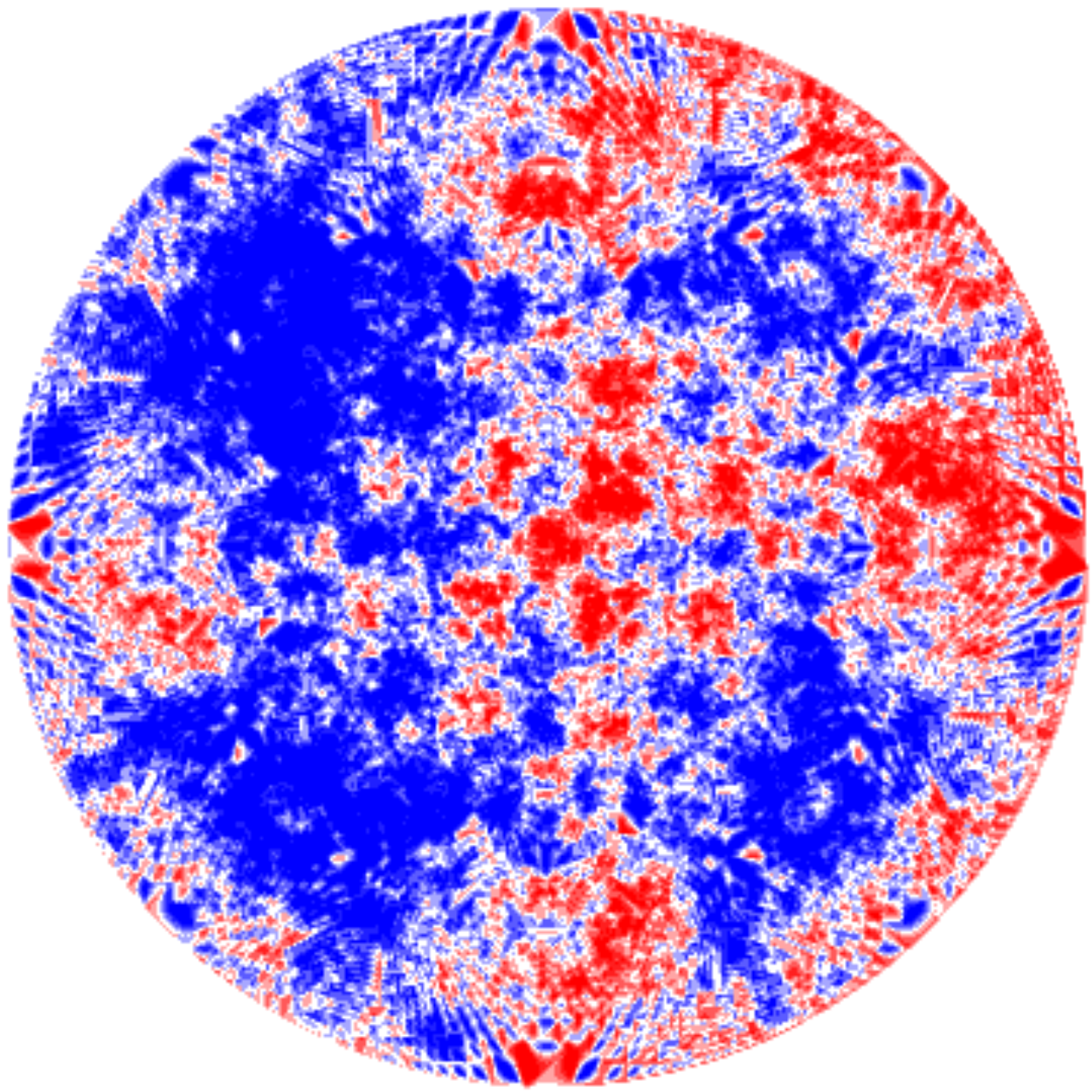}
    }
    \subfloat[$N=100100$.]{
        \drawblob{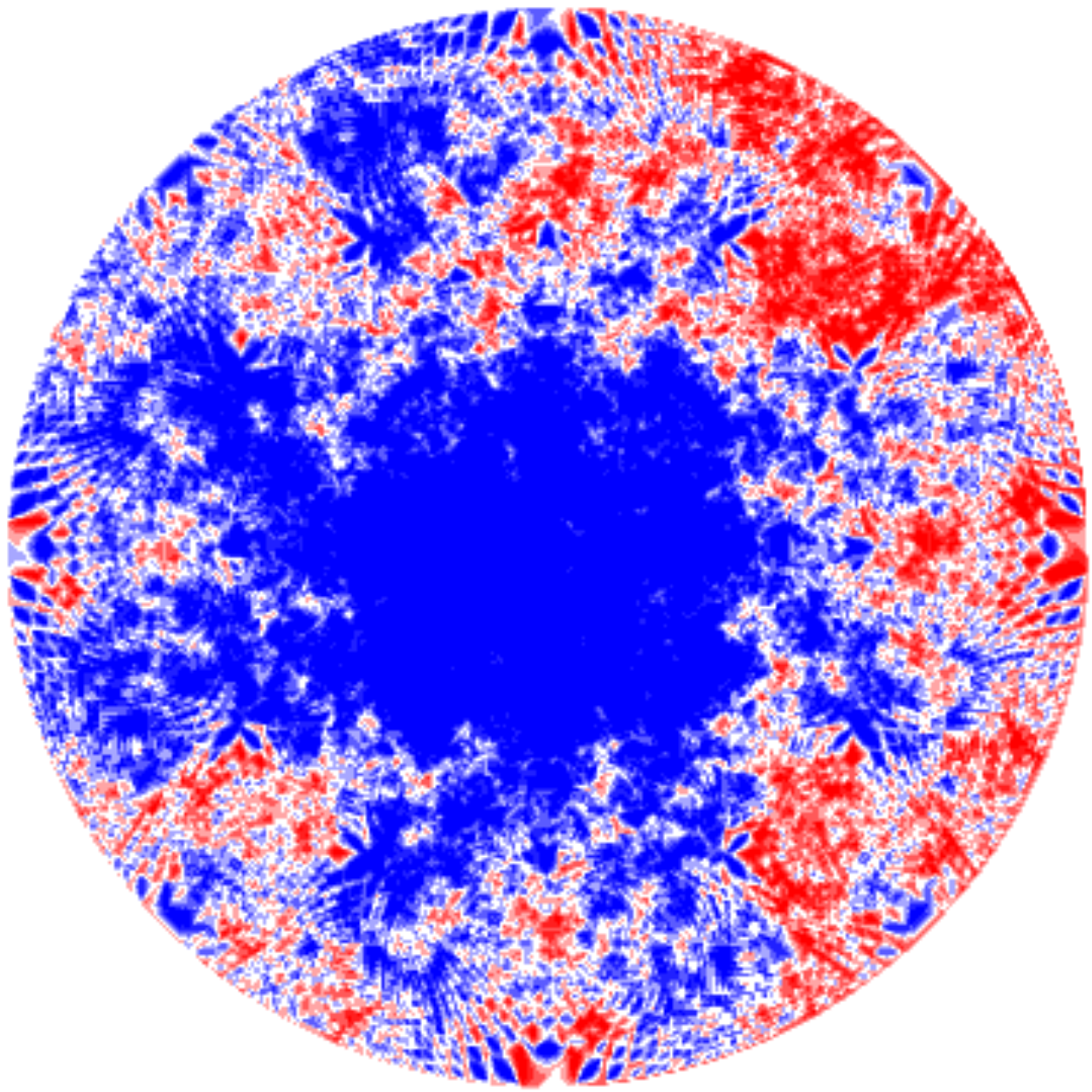}
    }
    \subfloat[$N=100200$.]{
        \drawblob{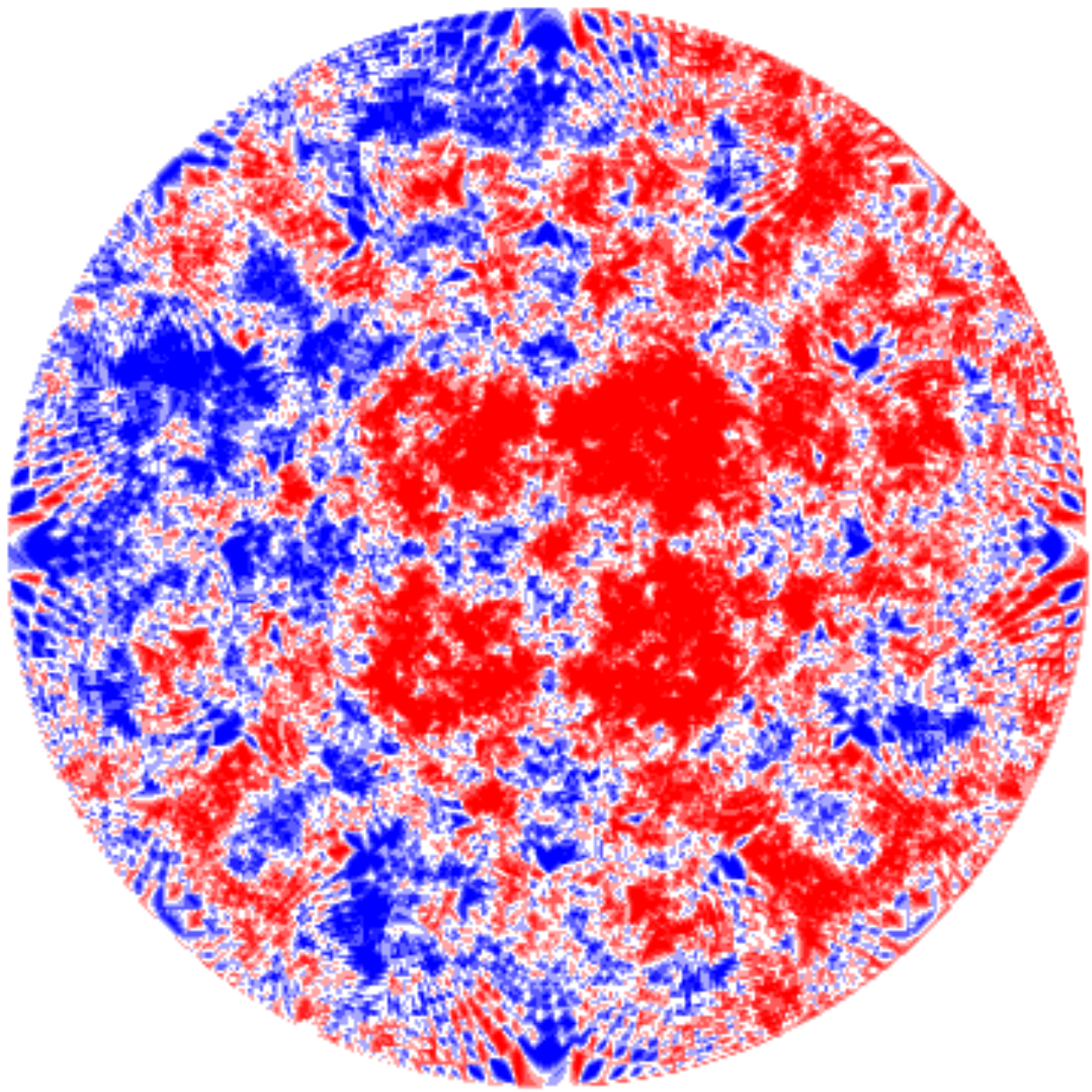}
    }
    \caption{Classic rotor router aggregation with counterclockwise rotor sequence.
        The pictures show the quality of the odometer approximation
        for different values of~$N$,
        as measured by the difference $u_1-u$.
        The site~$x$ is colored blue if $u_1(x)>u(x)$, red if $u_1(x)<u(x)$, and white if $u_1(x)=u(x)$.
        The dramatic dependence on~$N$ suggests that our approximation~$u_1$ captures substantially all of the large-scale regular structure in~$u$.}
    \label{fig:ododiff_afterapprox}
\end{figure}


\section{Experimental Results}
\label{sec:models}

We implemented our algorithm for three different growth models in~$\Z^2$: rotor-router aggregation, IDLA, and a hybrid of the two which we call ``low-discrepancy random stack.''  In this section we discuss some details of the implementation, comment on the observed runtime, and present our findings on the fluctuations of the cluster $A_N$ from circularity for large $N$.

As a basis for comparison to our algorithm, consider the time it takes to compute the occupied cluster $A_N$ for rotor-router aggregation by the traditional method of firing one vertex at a time.  If $z_1,\ldots,z_N \in \Z^2$ are the locations of the $N$ chips, define the \emph{quadratic weight} $Q(\mathbf{z}) = \sum_{i=1}^N |z_i|^2$, where $|(x,y)|=(x^2+y^2)^{1/2}$ is the Euclidean norm.  Firing a given vertex~$z$ four times results in exactly one chip being sent to each of the four neighbors $z\pm e_1$, $z\pm e_2$.  The net effect of these four firings on the quadratic weight is to increase~$Q$ by
	\[ |z+e_1|^2 + |z-e_1|^2 + |z+e_2|^2 + |z-e_2|^2 - 4|z|^2 = 4. \]
Thus, the total number of firings needed to produce the final occupied cluster~$A_N$ is approximately~$\sum_{z \in A_N} |z|^2$.  Since $A_N$ is close to a disk of area~$N$, this sum is about~$N^2/2\pi$.

Traditional step-by-step simulation therefore requires quadratic time to compute the occupied cluster.  Step-by-step simulation of IDLA also requires quadratic time, as observed in \cite{LBG,MM}.  We found experimentally that our algorithm ran in significantly shorter time: about $N \log N$ for the rotor-router model (\tabref{classic}), and about $N^{1.5}$ for IDLA (\tabref{IDLA}).

\subsection{Implementation details}
\label{sec:exp:implementation}

We implemented the described algorithm in C++.  The source code is available from~\cite{hugerotor}.
It is easy to compute the odometer approximation for $z$ with $|z|\geq100$ 
according to \eq{odoapprox}.  However, the odometer approximation for 
$z$ with $|z|<100$ is less straightforward as the McCrea-Whipple algorithm~\citep{MW} 
is numerically very ill-conditioned.  In order to avoid escalating errors with fixed precision
floating point numbers, 
we used the computer algebra system Maple  to precompute $a(z)$ as a rational linear combination of $1$ and $\frac{1}{\pi}$ for $|z|<100$.

For the annihilation step described in \secref{thealgorithm}, we used a multiscale approach to cancel out hills and holes efficiently.  More specifically, let $L_1,L_2,\ldots$ be an exponentially growing sequence of integers. For each $i \geq 1$ do
		\begin{itemize}
		\item Substep $i$: fire each hill / unfire each hole until it either cancels out or reaches a site in $G_i := (L_i \Z \times \Z) \cup (\Z \times L_i \Z)$.
		\end{itemize}
We used $L_1=1$ and $L_{i+1}=\lceil 1.9\,L_i \rceil$ for $i\geq1$.  Experimentally, the choice of $1.9$ resulted in the fastest run time.  During each substep $i$, we scan the grid and for each site $z \notin G_i$, if $z$ is a hill, fire it until it is no longer a hill; if $z$ is a hole, unfire it until it is no longer a hole.  We repeat this scanning procedure until no hills or holes remain outside of $G_i$.  The result is that a large number of hills and holes meet and cancel each other out, while the remainder are swept into the much sparser set $G_i$.  We then proceed to substep $i+1$, stopping when $L_i$ exceeds the diameter of the set of sites that absorb a chip.  At this stage we perform a final substep with $G_i = \emptyset$: in other words, repeatedly scan the grid, firing hills and unfiring holes with no restrictions on their location.  When no more hills or holes remain, we proceed to the reverse cycle-popping phase described in \secref{thealgorithm}. 

\begin{figure}[tbp]
    \centering
    \subfloat[$N=100000$.]{
        \drawblob{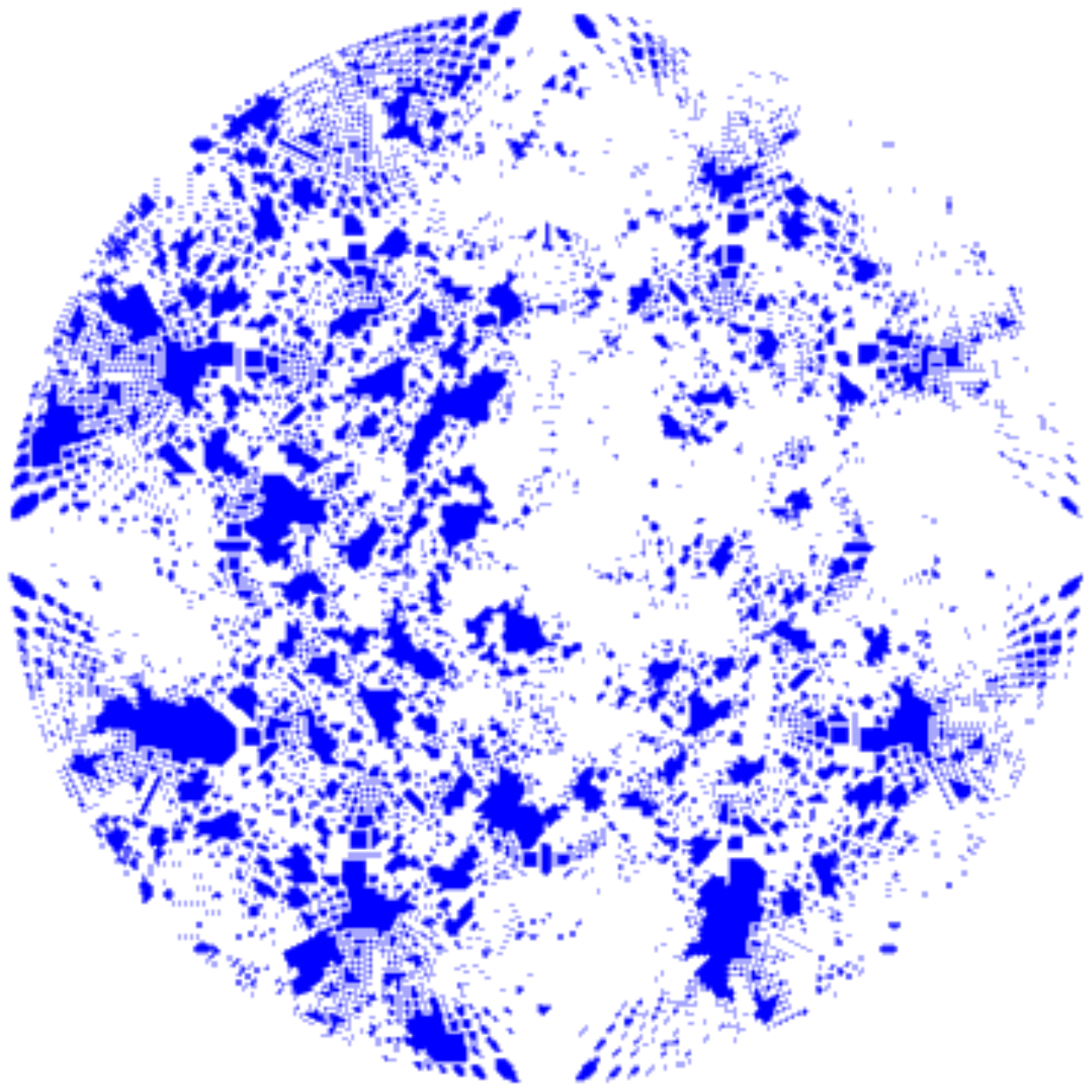}
    }
    \subfloat[$N=100100$.]{
        \drawblob{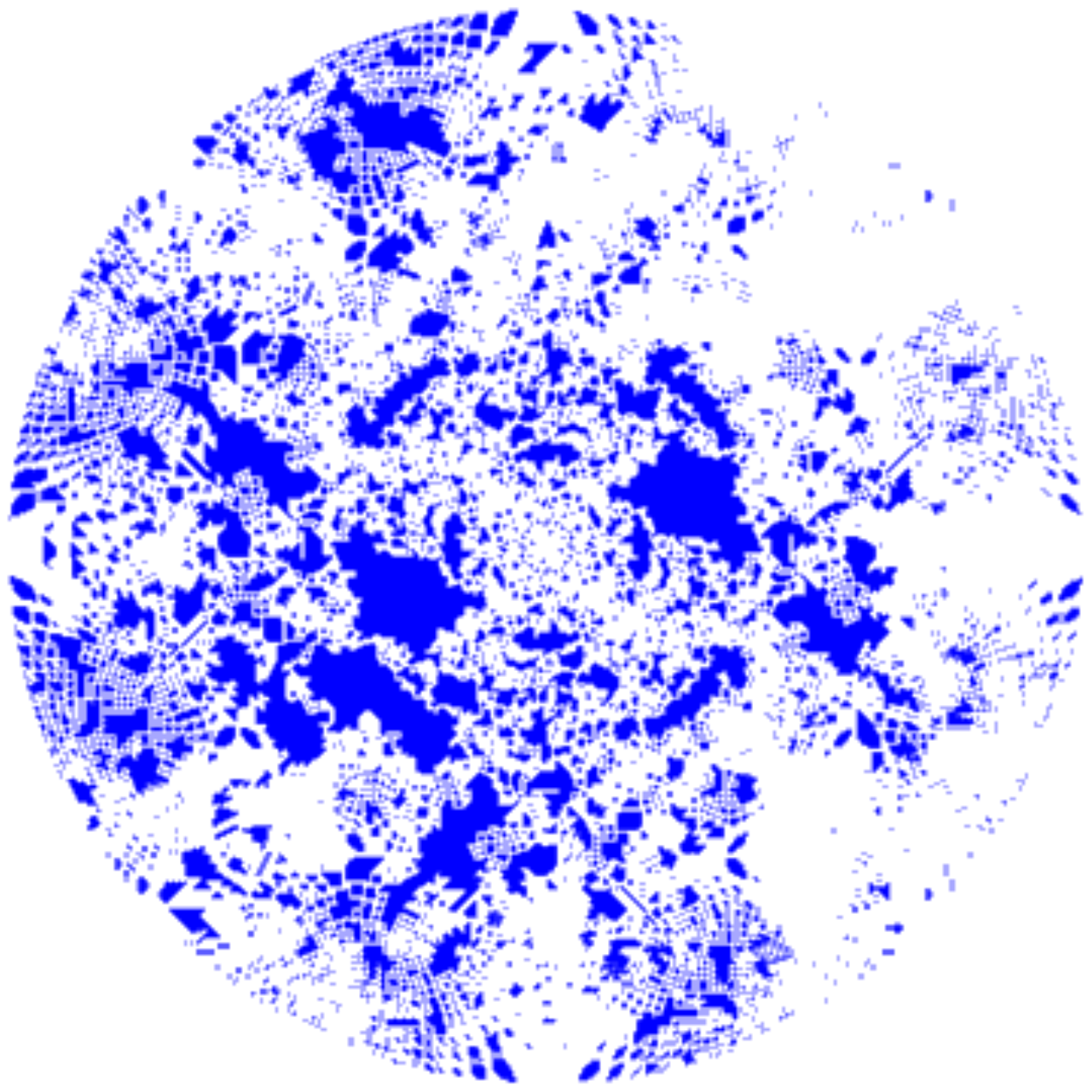}
    }
    \subfloat[$N=100200$.]{
        \drawblob{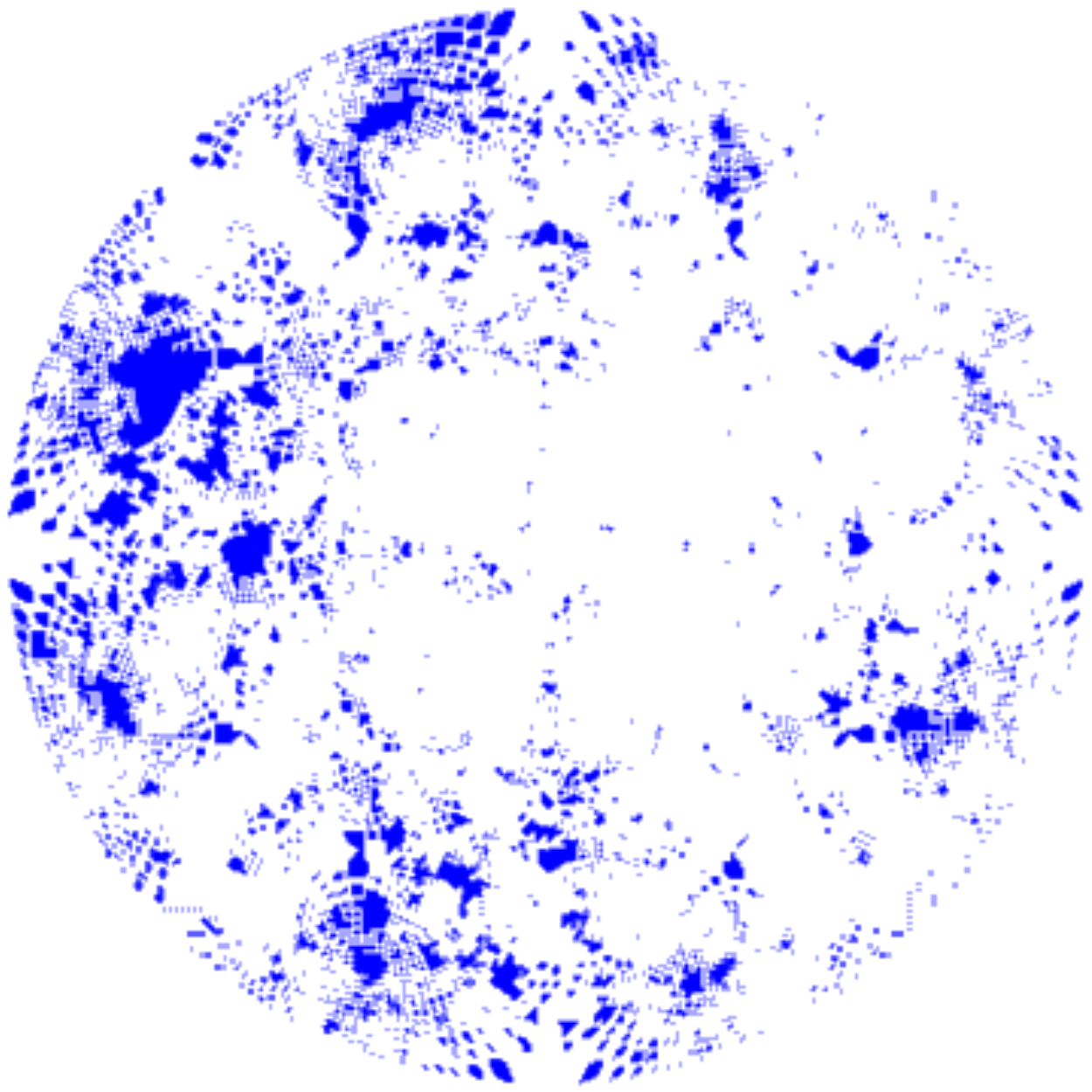}
    }
    \caption{Classic rotor router aggregation with counterclockwise rotor sequence.
        The pictures show the quality of the odometer approximation
        after the annihilation phase, for different values of~$N$, as measured by the difference $u_2-u$.
    The site~$x$ is colored blue if $u_2(x)>u(x)$, white if $u_2(x)=u(x)$.
    Note that after annihilation, there are no longer any sites satisfying $u_2(x)<u(x)$.
    The remaining odometer difference also shows how many cycles are then popped in
    the last phase of our algorithm.
    The darker the color, the more cycles run through this location.
    }
    \label{fig:ododiff_afterannihilation}
\end{figure}

Our rotor-router calculation (\textsection\ref{sec:exp:RR}) was performed on
a Fujitsu RX600S5 server with
four Xeon X7550 processors
and 2048 GB main memory.   Our IDLA calculations (\textsection\ref{sec:exp:IDLA}) were performed on a cluster of 96 Sun Fire V20z with AMD Opteron 250 processors.  For IDLA, our method depends strongly on the availability of a high-quality pseudorandom number generator.  
We used the cryptographically secure generator \emph{Advanced Encryption Standard} (AES)~\citep{AES}, which is the official successor of the well-known \emph{Data Encryption Standard} (DES).
We used a key size of $256$ bits with the
Rijndael cipher implementation by Rijmen, Bosselaers and Barreto,
which is also part of OpenSSH.  

We observed that C's built-in \texttt{rand($\,\cdot\,$)} function, which has a small period, produces a noticeably smaller difference between inradius and outradius (about $13\%$ smaller for $N=2^{10}$).
We did not pursue this further to study whether this difference persists for larger values of $N$.

\subsection{Rotor-router aggregation}
\label{sec:exp:RR}

In the classic rotor-router model, the rotor stack is the cyclic sequence
of the four cardinal directions in counterclockwise order.
\tabref{classic} shows some statistical data of our computation.
The absolute error in our odometer approximation \[ \|u_1 - u\|_1 = \sum_x |u_1(x)-u(x)| \] appears to scale linearly with $N$.  This quantity is certainly a lower bound for the running time of our algorithm.  The measured runtimes indicate close-to-linear runtime behavior, which suggests that our multiscale approach to canceling out hills and holes is relatively efficient.

\figref{ododiff_afterapprox} depicts the odometer difference $u_1(x)-u(x)$ for three different values of~$N$.  \figref{ododiff_afterannihilation} depicts the odometer difference $u_2(x)-u(x)$ after the annihilation step of the algorithm.

\begin{table*}[tb]
\begin{center}
\scalebox{0.71}{
\begin{tabular}{r@{=}lr@{ }lccc@{}c@{\!}c@{\!}c}
\hline
\hline
\addlinespace[1mm]
\multicolumn{2}{c}{\bf Number of } &
\multicolumn{2}{c}{\multirow{2}{*}{\bf Runtime}} &
\multicolumn{2}{c}{\bf \underline{Radius Difference}} &
\multicolumn{1}{c}{\multirow{2}{*}{$\bm{\|u_1-u\|_1 \big/ N}$}} &
\multicolumn{1}{c}{\multirow{2}{*}{$\bm{\max |u_1-u|}$}} &
\multicolumn{1}{c}{\bf highest } &
\multicolumn{1}{c}{\bf deepest }
\\
\multicolumn{2}{c}{\bf chips $\bm N$} &
&&
 \textbf{absolute} &
 \textbf{recentered} &
&
&
\bf hill &
\bf hole \\
\addlinespace[1mm]
\hline
\hline
\addlinespace[1mm]

$2^{10}$ &          1,024 & 1.60 & ms    & 1.324 & 0.278 & 1.800 &   6 & 3 & -1 \\ 
$2^{12}$ &          4,096 & 2.58 & ms    & 1.523 & 0.138 & 3.370 &  10 & 3 & -1 \\ 
$2^{14}$ &         16,384 & 5.71 & ms    & 1.579 & 0.166 & 2.417 &  12 & 3 & -1 \\ 
$2^{16}$ &         65,536 & 21.5 & ms    & 1.611 & 0.429 & 4.461 &  17 & 3 & -1 \\ 
$2^{18}$ &        262,144 & 67.1 & ms    & 1.565 & 0.346 & 2.919 &  16 & 3 & -1 \\ 
$2^{20}$ &      1,048,576 & 0.26 & sec   & 1.642 & 0.362 & 4.323 &  23 & 3 & -1 \\ 
$2^{22}$ &      4,194,304 & 1.04 & sec   & 1.596 & 0.316 & 4.220 &  29 & 3 & -1 \\ 
$2^{24}$ &     16,777,216 & 3.53 & sec   & 1.614 & 0.396 & 3.974 &  45 & 3 & -1 \\ 
$2^{26}$ &     67,108,864 & 0.24 & min   & 1.658 & 0.368 & 4.695 &  62 & 3 & -1 \\ 
$2^{28}$ &    268,435,456 & 0.98 & min   & 1.639 & 0.340 & 4.463 &  83 & 3 & -1 \\ 
$2^{30}$ &  1,073,741,824 & 4.04 & min   & 1.635 & 0.414 & 4.309 &  91 & 3 & -1 \\ 
$2^{32}$ &  4,294,967,296 & 0.28 & hours & 1.650 & 0.366 & 4.383 & 172 & 4 & -2 \\ 
$2^{34}$ & 17,179,869,184 & 1.10 & hours & 1.688 & 0.439 & 4.734 & 252 & 11 & -8 \\ 
$2^{36}$ & 68,719,476,736 & 3.80 & hours & 1.587 & 0.385 & 5.408 & 353 & 38 & -35 \\

\hline \hline
\end{tabular}}
\end{center}
\caption{Simulation results for classic rotor-router aggregation with counterclockwise rotor sequence.
    The given runtime is the total runtime of the calculation of one rotor-router
    aggregation of the given size on a Fujitsu RX600S5 server.
    The next two columns show the difference between the outradius and inradius
    of the occupied cluster $A_N$, measured with respect to the origin
    (``absolute'') and with respect to the putative center of mass
    $\big(\frac12, \frac12\big)$ (``recentered'').  The next two columns give
    two measurements of the error of our odometer approximation~$u_1$, the total
    absolute error and maximum pointwise error. In the last two columns,
    ``highest hill'' and ``deepest hole'' refer respectively to $\max_x
    \sigma_1(x)$ and $\min_x \sigma_1(x)$.}
\label{tab:classic}
\end{table*}

The asymptotic shape of rotor-router aggregation is a disk~\citep{LP08,LP09a}.  To measure how close $A_N$ is to a disk, we define the \emph{inradius} and \emph{outradius} of a set $A \subset \Z^2$ by
\[
    \inrad(A) = \min \{|x|\colon x \notin A \}
\]
and
\[
    \outrad(A) = \max \{|x|\colon x \in A\}.
\]
We then define
\[
    \diff(N) = \outrad(A_N) - \inrad(A_N).
\]
A natural question is whether this difference is bounded independent of $N$.	We certainly expect it to  increase much more slowly than the order $\log N$ observed for IDLA.

\citet{Kleber} calculated that $\diff(3\cdot 10^6)\approx1.6106$.  We can now extend the measurement of $\diff(N)$ up to $N = 2^{36} \approx 6.8 \cdot 10^{10}$ (\tabref{classic}, third column).
Our algorithm runs in less than four hours for this value of~$N$; by comparison, a step-by-step simulation of this size would take about $23000$ years on a computer with one billion operations per second.
In our implementation, the limiting factor is memory rather than time.

Up to dihedral symmetry, there are three different balanced period\nobreakdash-$4$ rotor sequences for~$\Z^2$: \rotseq{WENS}, \rotseq{WNSE}, and \rotseq{WNES}.  The notation $\rotseq{WENS}$ means that the first four rotors in each stack point respectively west, east, north and south.

\figref{radiusdiff:origin} shows the radius difference $\diff(N)$ for various $N$
for the three different rotor sequences.  As these values are rather noisy, we have also calculated and plotted
the averages
\begin{equation}
    \label{eq:avgdiff}
    \avgdiff(N) := \frac{1}{|I(N)|} \sum_{N' \in I(N)} \diff(N')
\end{equation}
with
\[
    I(N) = \begin{cases}
            \Big[\tfrac{N}{2},\tfrac{3N}{2}\Big] & \text{for $N\leq10^6$,}\\
            [N-5\cdot 10^5, N+5\cdot 10^5] & \text{for $N>10^6$.}
        \end{cases}
\]
Note that in \figref{radiusdiff:origin}, the radius difference
$\avgdiff(N)$ grows extremely slowly in~$N$.
In particular, it appears to be sublogarithmic.

\begin{figure}[tbp]
    \centering
    \subfloat[Radius difference around the origin.]{
        \label{fig:radiusdiff:origin}
        \includegraphics[width=.48\textwidth,clip]{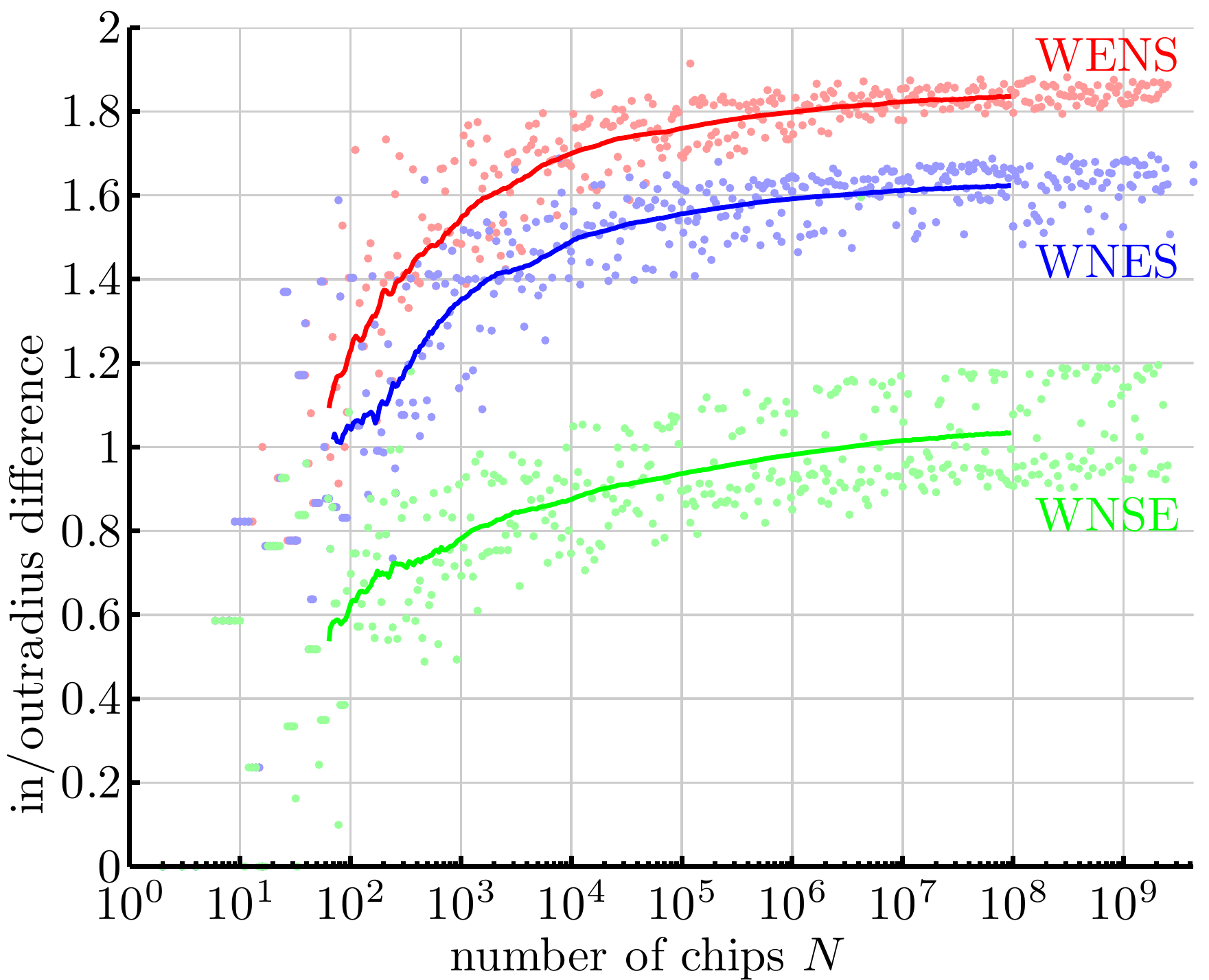}
    }
    \subfloat[Radius difference around putative center.]{
        \label{fig:radiusdiff:centered}
        \includegraphics[width=.48\textwidth,clip]{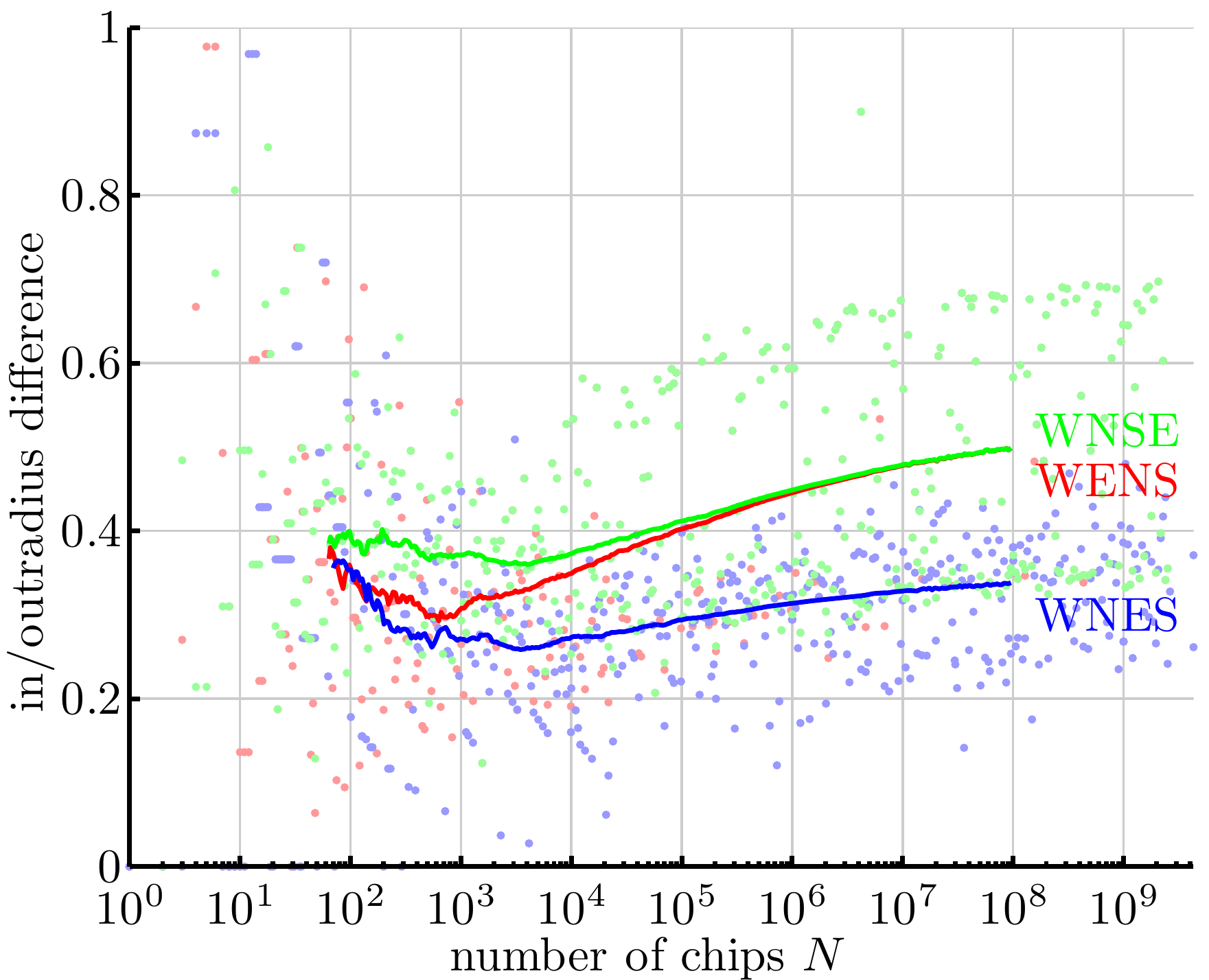}
    }
    \caption{Difference between the inradius and outradius of the rotor-router aggregate,
        for different numbers of chips~$N$.
        The single dots are individual values of $\diff(N)$ (left) and $\diff'(N)$ (right).
        The darker curves show the averages $\avgdiff(N)$ and $\avgdiff'(N)$
        as defined in \eqs{avgdiff}{avgdiffp}.
        (Color scheme: \rotseq{WNES}=blue, \rotseq{WNSE}=green, \rotseq{WENS}=red.)}
    \label{fig:radiusdiff}
\end{figure}



We observe a systematic difference in behavior for the three different rotor sequences.
The observed radius differences are lowest for \rotseq{WNSE}, intermediate for \rotseq{WNES}, and highest for \rotseq{WENS}.  For example,
\[
    \avgdiff(10^8)\approx
        \begin{cases}
            1.034 & \text{for \rotseq{WNSE},}\\ 
            1.623 & \text{for \rotseq{WNES},}\\ 
            1.837 & \text{for \rotseq{WENS}.}\\ 
        \end{cases}
\]
This difference can be partially explained by considering the center of mass of the aggregate.  Recall that our convention is ``retrospective'' (as opposed to ``prospective'') rotor
notation: that is, the rotor currently on top of the stack indicates where the last chip has
gone rather than where the next chip will go.
Hence for \rotseq{WNES} rotors, the first time each site fires it sends a chip north, the next time east, then south, then west.
As about $1/4$ of the sites end up
in each of the four rotor states, for \rotseq{WNES} rotors about half of the sites send
one more chip N than S, and (a different but overlapping)
half send one more chip E than W.
As a result, the center of mass of the set of occupied sites is close to $(1/2, 1/2)$.
For \rotseq{WENS} the center of mass is close to $(3/4,1/4)$,
and for \rotseq{WNSE} it's close to $(1/4,1/4)$.

In some sense, a better measure of circularity than $\diff(N)$ is the radius difference relative to the center of mass.  Thus we define
	\[ \diff'(N) = r_{out}(A_N-\mathbf{c}) - r_{in}(A_N-\mathbf{c}) \]
where $\mathbf{c}$ is one of $(1/2, 1/2)$, $(3/4,1/4)$, or $(1/4,1/4)$ chosen according to the rotor sequence used.  Let
\begin{equation}
    \label{eq:avgdiffp}
    \avgdiff'(N) := \frac{1}{|I(N)|} \sum_{N' \in I(N)} \diff'(N).
\end{equation}
These values are plotted for various~$N$ in \figref{radiusdiff:centered}.
We find
\[
    \avgdiff\,'(10^8)\approx
        \begin{cases}
            0.499 & \text{for \rotseq{WNSE} and \rotseq{WENS},}\\ 
            0.338 & \text{for \rotseq{WNES}.}\\ 
        \end{cases}
\]
The differences are now significantly smaller, and the two non-cyclic rotor sequences \rotseq{WNSE} and \rotseq{WENS} have nearly the same radius difference for large $N$.   To see why, note that \rotseq{WENS} is obtained from \rotseq{WNSE} by a shift in the stacks (to \rotseq{EWNS}) followed by interchanging the directions east and west.  Thus the observed difference in $\diff(N)$ between these two rotor sequences is entirely due to the effect of the initial condition of rotors primed to send chips west.  By adjusting for the center of mass, we have largely removed this effect in $\diff'(N)$.


\subsection{Internal Diffusion Limited Aggregation (IDLA)}
\label{sec:exp:IDLA}
In IDLA, the rotor directions~$\rho_k(x)$ for $x\in\Z^2$ and $k\in\Z$ are chosen independently and uniformly at random from among the four cardinal directions.
In the course of firing and unfiring during steps~2 and~3 of our algorithm, the same rotor $\rho_k(x)$ may be requested several times.  Therefore, we need to be able to generate the same pseudorandom value for $\rho_k(x)$ each time it is used.
Generating and storing all rotors $\rho_k(x)$ for all $x$ and all $1 \leq k \leq u_1(x)$ is out of the question, however, since it would cost $\Omega(N^2)$ time and space.

\citet{MM} encountered the same issue in developing a fast parallel algorithm for IDLA.  Rather than store all of the random choices, they chose to store only certain seed values for the random number generator and generate random walk steps online as needed.  Next we describe how to adapt this idea to our setting for fast serial computation of IDLA.

\begin{table*}[tb]
\begin{center}
\scalebox{0.71}{
\begin{tabular}{r@{=}lr@{ }lr@{$\pm$}lr@{$\pm$}lr@{$\pm$}lr}
\hline
\hline
\addlinespace[1mm]
\multicolumn{2}{c}{\bf Number of } &
\multicolumn{2}{c}{\bf Average} &
\multicolumn{2}{c}{\bf Radius} &
\multicolumn{2}{c}{\multirow{2}{*}{\bf $\bm{\|u_1-u\|_1/N^{3/2}}$}} &
\multicolumn{2}{c}{\multirow{2}{*}{\bf $\bm{\max |u_1-u|}$}} &
\multicolumn{1}{c}{\bf Number}
\\
\multicolumn{2}{c}{\bf chips $\bm N$} &
\multicolumn{2}{c}{\bf Runtime} &
\multicolumn{2}{c}{\bf Difference} &
\multicolumn{2}{c}{} &
\multicolumn{2}{c}{} &
\multicolumn{1}{c}{\bf of runs}
\\
\addlinespace[1mm]
\hline
\hline
\addlinespace[1mm]

$2^{10}$ &       1,024 & 9.80 & ms    & 3.198 & 0.569 & 0.490 & 0.057 &     134 &     27 &  $10^{6}$ \\ 
$2^{11}$ &       2,048 & 26.9 & ms    & 3.569 & 0.547 & 0.516 & 0.054 &     220 &     41 &  $10^{6}$ \\ 
$2^{12}$ &       4,096 & 73.9 & ms    & 3.948 & 0.553 & 0.541 & 0.051 &     355 &     62 &  $10^{6}$ \\ 
$2^{13}$ &       8,192 & 0.21 & sec   & 4.307 & 0.556 & 0.565 & 0.049 &     568 &     93 &  $10^{6}$ \\ 
$2^{14}$ &      16,384 & 0.62 & sec   & 4.664 & 0.566 & 0.588 & 0.047 &     901 &    139 &  $10^{6}$ \\ 
$2^{15}$ &      32,768 & 1.81 & sec   & 5.027 & 0.578 & 0.610 & 0.045 &   1,418 &    207 &  $10^{6}$ \\ 
$2^{16}$ &      65,536 & 5.29 & sec   & 5.393 & 0.578 & 0.631 & 0.043 &   2,216 &    307 &  $10^{6}$ \\ 
$2^{17}$ &     131,072 & 0.26 & min   & 5.763 & 0.584 & 0.652 & 0.042 &   3,443 &    456 &  $10^{5}$ \\ 
$2^{18}$ &     262,144 & 0.76 & min   & 6.125 & 0.588 & 0.673 & 0.041 &   5,317 &    672 &  $10^{5}$ \\ 
$2^{19}$ &     524,288 & 2.26 & min   & 6.493 & 0.593 & 0.692 & 0.039 &   8,179 &    985 &  $10^{5}$ \\ 
$2^{20}$ &   1,048,576 & 6.74 & min   & 6.858 & 0.594 & 0.711 & 0.038 &  12,522 &  1,455 &  $10^{5}$ \\ 
$2^{21}$ &   2,097,152 & 0.34 & hours & 7.222 & 0.600 & 0.730 & 0.038 &  19,085 &  2,131 & $6 \cdot 10^{4}$ \\ 
$2^{22}$ &   4,194,304 & 1.01 & hours & 7.596 & 0.600 & 0.748 & 0.036 &  29,007 &  3,109 & $6 \cdot 10^{4}$ \\ 
$2^{23}$ &   8,388,608 & 3.95 & hours & 7.968 & 0.601 & 0.767 & 0.036 &  44,007 &  4,471 & $3 \cdot 10^{3}$ \\ 
$2^{24}$ &  16,777,216 & 14.9 & hours & 8.319 & 0.605 & 0.783 & 0.035 &  66,418 &  6,763 & $3 \cdot 10^{3}$ \\ 
$2^{25}$ &  33,554,432 & 44.3 & hours & 8.699 & 0.575 & 0.801 & 0.033 &  99,667 & 10,192 & $4 \cdot 10^{2}$ \\

\hline \hline
\end{tabular}}
\end{center}
\caption{Simulation results for IDLA.  
    The given runtime is the total time taken for the calculation of one IDLA
    cluster of the given size on a single core.
    To fit within 4 GB (8 GB for $N=2^{25}$) main memory, we used
    $\lambda=0$ for $N\leq2^{22}$,
    $\lambda=2$ for $N=2^{23}$,
    $\lambda=5$ for $N\geq 2^{24}$.
    The next column shows the difference between the outradius and inradius of the occupied cluster $A_N$.
    The fourth and fifth columns give two measurements of the error of our odometer approximation~$u_1$,
    the total absolute error and maximum pointwise error.
     The values shown are averages and standard deviations over many independent trials;  the last column shows the number of trials.
}
\label{tab:IDLA}
\end{table*}

The AES pseudorandom number generator takes as input a block of 128 bits and ``encrypts'' it, outputting a block of 128 pseudorandom bits.  We interpret the output block as the binary expansion of a number in the interval $[0,1)$.
Let $\rnd(b)$ be the pseudorandom number generated from input block~$b$.
Let
	\[ U_k(x) = \rnd(\block(x,k,a)), \]
where $\block(x,k,a)$ is a simple deterministic function that assumes distinct values 
for each triple $(x,k,a)$ of site~$x$, odometer value~$1 \leq k \leq K$, and integer $1 \leq a \leq A$.  The integer~$a$ is fixed for each run of the algorithm, and~$A$ is the total number of runs of the algorithm; this way, each run generates an independent IDLA cluster.  The bound~$K$ is chosen to be safely larger than the maximal odometer value~$u_1(o) \approx 2 r^2 \ln r$.

Writing $\uparrow$, $\rightarrow$, $\downarrow$, $\leftarrow$
for the four outgoing edges from site $x$, we set
\begin{equation}
    \label{eq:IDLArho}
    \rho_k(x):=
    \begin{cases}
    \uparrow    & \text{ if }  0  \leq U_k(x) < 1/4, \\
    \rightarrow & \text{ if } 1/4 \leq U_k(x) < 1/2, \\
    \downarrow  & \text{ if } 1/2 \leq U_k(x) < 3/4, \\
    \leftarrow  & \text{ if } 3/4 \leq U_k(x) < 1. \\
    \end{cases}
\end{equation}

The first step of the algorithm described in \secref{thealgorithm} is to calculate~$\sigma_1$ from the odometer approximation~$u_1$.  In this calculation, the definition of $R(e,n)$ given in \eq{rotorfrequencies} involves evaluating $\rho_k(x)$
for all $1 \leq k \leq n$.  As this is much too expensive, we instead use the fact that $R(e,n)$ is a random variable with the $\Binomial(n,1/4)$ distribution.  In steps 2 and 3 of the algorithm, we need to sample some individual rotors $\rho_k(x)$, but typically not too many: on the order of $\sqrt{u_1(x)}$.  The distribution of these rotors depends on the binomials already drawn.  We think of first populating an urn with balls of $4$ colors corresponding to the directions $\uparrow$, $\rightarrow$, $\downarrow$, $\leftarrow$.  When the algorithm asks for an individual rotor, we draw a ball at random from the urn using our knowledge of how many balls of each color remain.

This approach works well for small and moderate system sizes, but for large~$N$ it is too memory-intensive.  The memory usage comes from the need to store the rotors previously drawn in order to keep track of how many balls of each color remain in the urn.  Note that keeping a count does not suffice, because the algorithm may request a single rotor multiple times.

\begin{figure}[tbp]
    \centering
    \subfloat[Sample variance of the first $100$ moments.]{
        \label{fig:moments:1}
        \includegraphics[height=6.6cm,clip]{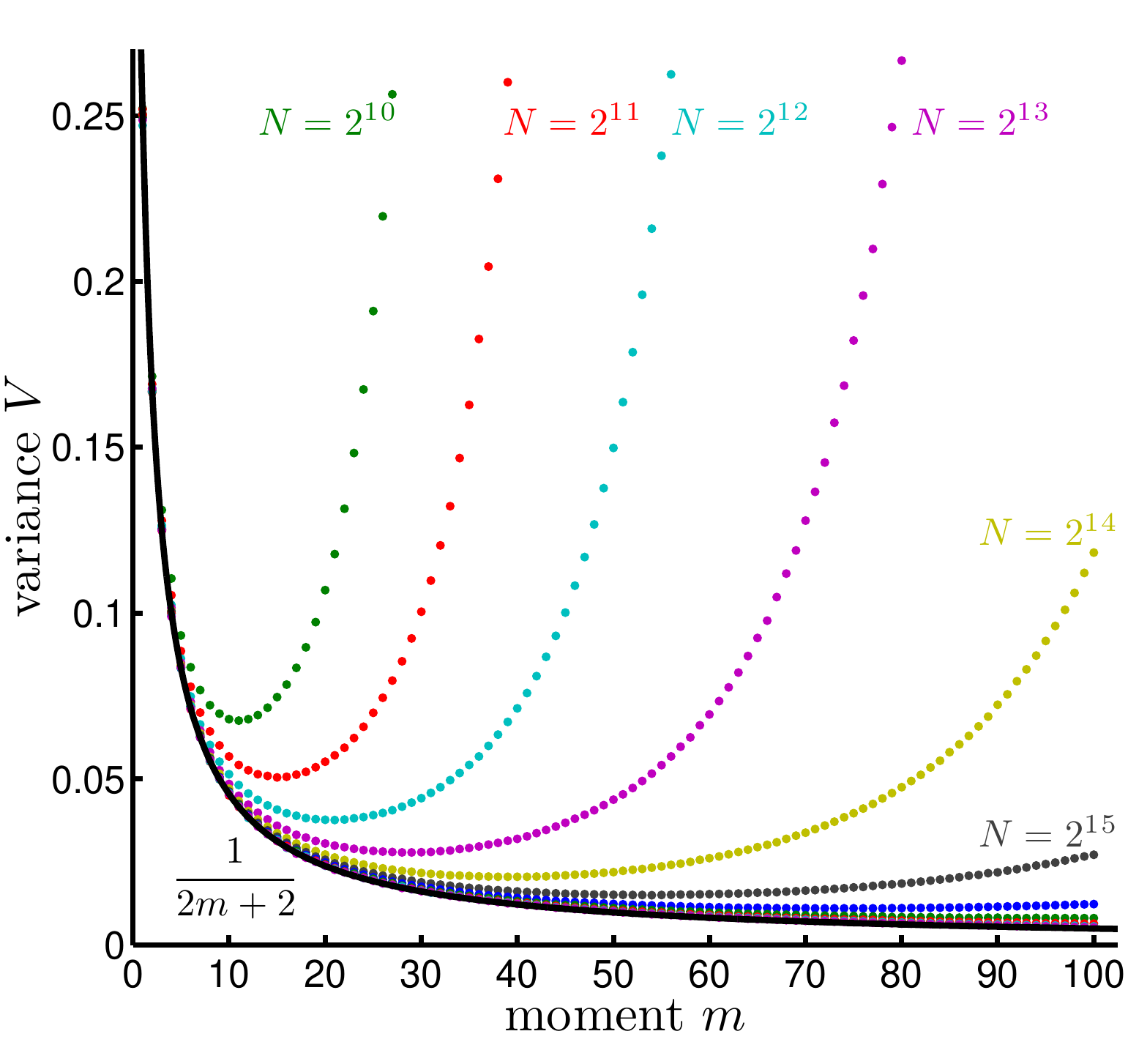}
    }
    \subfloat[Histograms of the first three moments.]{
        \label{fig:moments:hist}
        \includegraphics[height=6.6cm,clip]{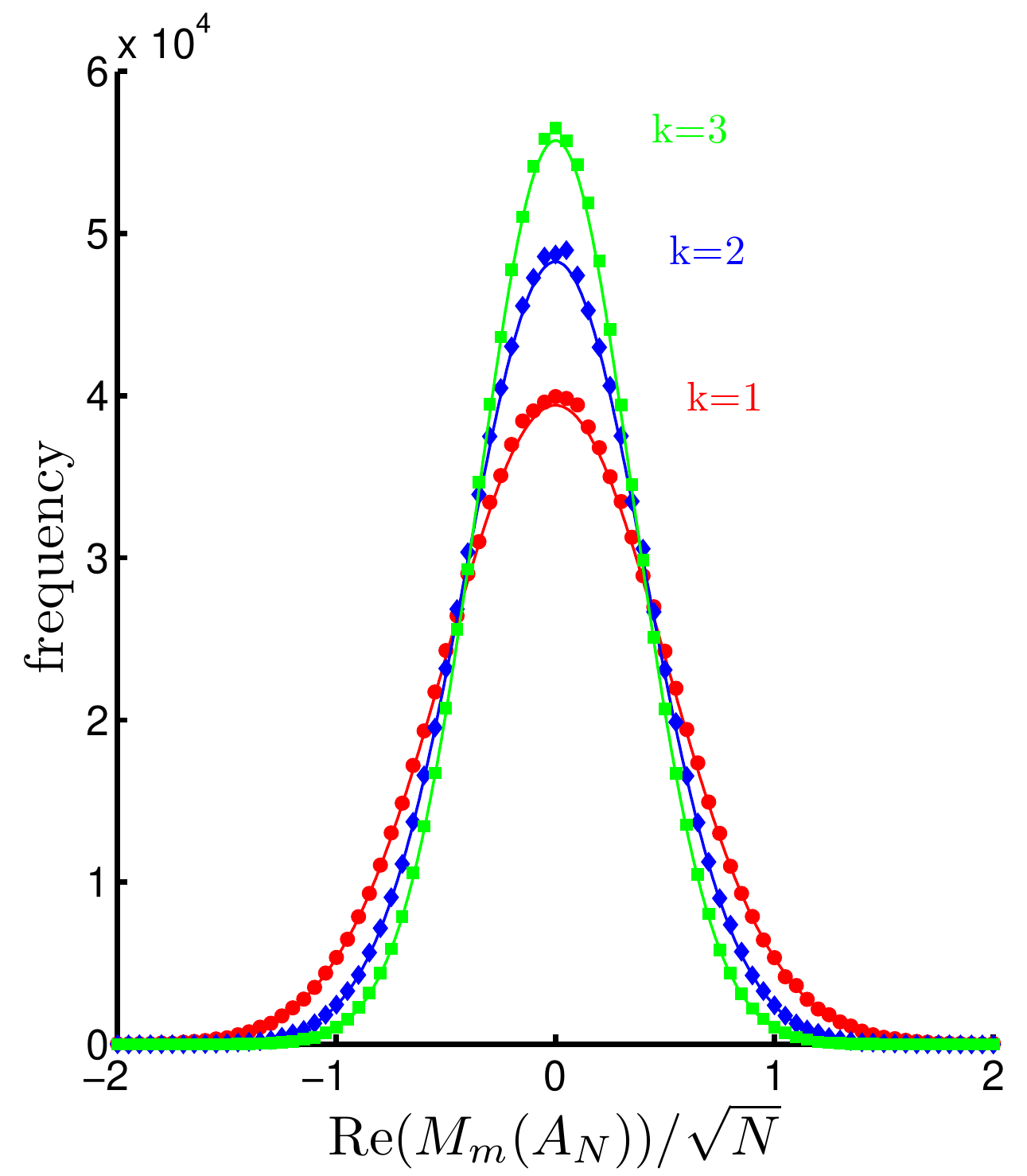}
    }
    \caption{Complex moments of the IDLA cluster.  Left: The sample variance $V(m) = \EE\, \Re (M_m(A_{N})/\sqrt{N})^2$ of the real parts of the first $100$ moments, for $N=2^{10},\ldots,2^{20}$.  As $N$ increases, the variance of the real part of the $m$-th moment approaches $1/(2m+2)$, in agreement with the results of~\cite{JLS2}.  Right: Histogram of the real part of the first three moments for $N=2^{16}$.
    The histogram shows 1,000,000 independent runs in bins of size 0.05.
    Data for the imaginary parts is similar.}
    \label{fig:moments}
\end{figure}

Fix a parameter $\lambda \geq 0$ representing the tradeoff between time and memory.  A larger value of $\lambda$ will result in saving memory at the cost of additional time.  Let
	\[ f(x) = \left(u_1(x) - \lambda \sqrt{u_1(x)} \right)^+. \]
For each site $x$ with $f(x)>0$, we sample three binomial random variables
\begin{align*}
    B &\sim \Binomial\big(f(x) ,\,\tfrac14\big), \\
	B' &\sim \Binomial\big(f(x)-B,\,\tfrac13\big), \\
	B'' &\sim \Binomial\big(f(x)-B-B',\,\tfrac12\big).
\end{align*}
We then set
\begin{align*}
    R(\uparrow,f(x)) &= B \\
    R(\rightarrow,f(x)) &= B' \\
    R(\downarrow,f(x)) &= B'' \\
    R(\leftarrow,f(x)) &= f(x)-B-B'-B''.
\end{align*}

Next, to implement step~$1$ of the algorithm described in \secref{thealgorithm}, we need to know $R(e,u_1(x))$.  So we compute
	\[ R(e,u_1(x)) = R(e,f(x)) + \# \{f(x) < k \leq u_1(x) \mid \rho_k(x)=e \}. \]
Note that if $\lambda$ is large, then this calculation is expensive in time, since it involves calling the pseudorandom number generator to draw as many as $\lambda \sqrt{u_1(x)}$ rotors
	\[ \rho_k(x),  \qquad f(x) < k \leq u_1(x) \]
using equation (\ref{eq:IDLArho}).  But, crucially, these rotors do not need to be stored.

During steps~2 and~3 of the algorithm, we sample any rotors $\rho_k(x)$ for $k>f(x)$ as needed using (\ref{eq:IDLArho}).
Rotors $\rho_k(x)$ for $k \leq f(x)$ can be sampled online as needed according to the distribution
\begin{equation*}
    \rho_{k}(x):=
    \begin{cases}
    \uparrow    & \text{ if }
    		    U_k(x) \in \Big[0,
                      \tfrac{R(\uparrow,k)}{k} \Big), \\[1.2mm]
    \rightarrow & \text{ if }
                      U_k(x) \in \Big[ \tfrac{R(\uparrow,k)}{k},
                      \tfrac{R(\uparrow,k)+R(\rightarrow,k)}{k} \Big), \\[1.2mm]
    \downarrow  & \text{ if }
    		    U_k(x) \in \Big[
                      \tfrac{R(\uparrow,k)+R(\rightarrow,k)}{k},
                      \tfrac{R(\uparrow,k)+R(\rightarrow,k)+R(\downarrow,k)}{k} \Big), \\[1.2mm]
    \leftarrow  & \text{ if }
    		    U_k(x) \in \Big[
                      \tfrac{R(\uparrow,k)+R(\rightarrow,k)+R(\downarrow,k)}{k},1 \Big).
    \end{cases}
\end{equation*}
Initially, the values $R(e,k)$ are known only for $k=f(x)$.  We generate the rotors $\rho_k(x)$ as needed in order of decreasing index~$k$, starting with $k=f(x)$.  Upon generating a new rotor $\rho_k(x)=e$, we inductively set
\[
    R(e,k-1) = R(e,k) -1
\]
and $R(e',k-1)= R(e',k)$ for
$e' \in \{\uparrow,\rightarrow,\downarrow,\leftarrow\} - \{e\}$.
These values specify the distribution for the next rotor $\rho_{k-1}(x)$
in case it is needed later.

The results of our large-scale simulations of IDLA are summarized in \tabref{IDLA}, extending the experiments of~\citet{MM} ($N \leq 10^{5.25}$ with $100$ trials) to over $10^6$ trials for $N \leq 2^{16}$ and over $300$ trials for $N \leq 2^{25} \approx 10^{7.5}$.
 The observed runtime of our algorithm for IDLA is about $N^{1.5}$;
 in contrast, building an IDLA cluster of size $N$ by serial simulation of $N$ random walks takes expected time order~$N^2$ (cf.~\cite[Fig.~3]{MM}).

\begin{figure}[tbp]
    \centering
    \subfloat[IDLA.]{
        \label{fig:radiusdiff2:IDLA}
        \includegraphics[width=.48\textwidth,clip]{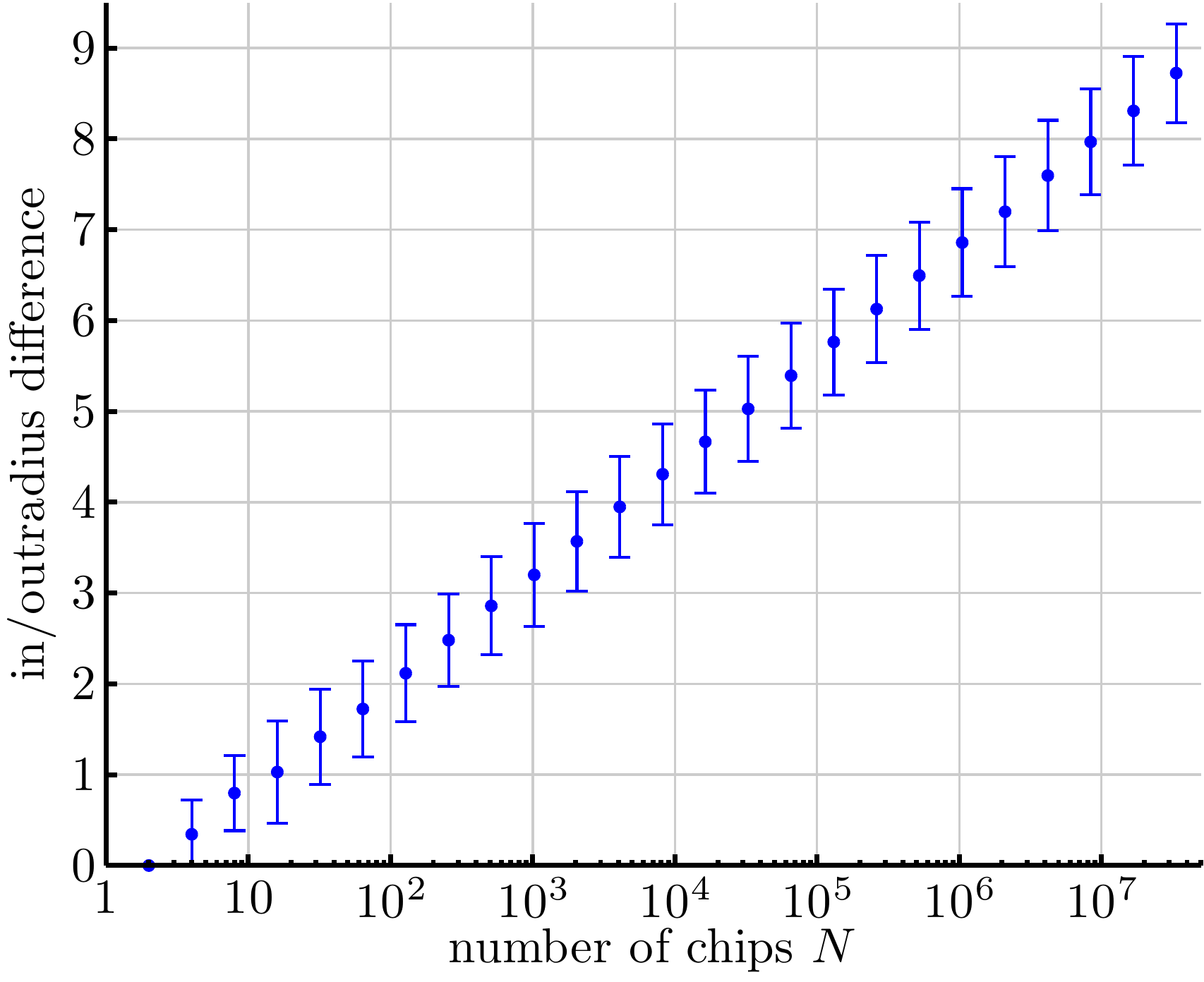}
    }
    \subfloat[Low-discrepancy random stack.]{
        \label{fig:radiusdiff2:rndstack}
        \includegraphics[width=.48\textwidth,clip]{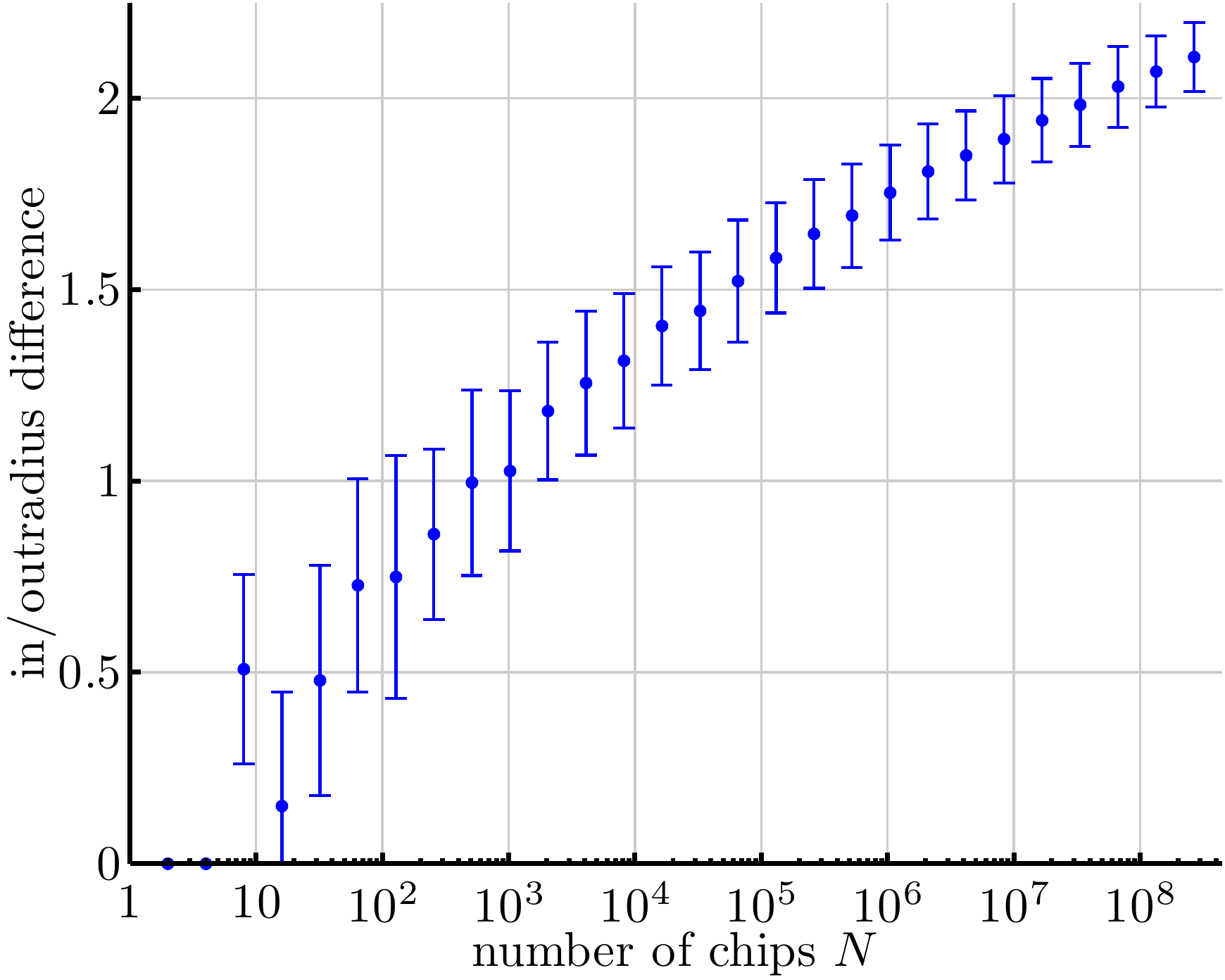}
    }
    \caption{Difference between inradius and outradius for different numbers of chips $N$
        for IDLA (\textsection\ref{sec:exp:IDLA}) and
        the low discrepancy random stack model (\textsection\ref{sec:exp:rndstack}).  Dots indicate means and error bars indicate standard deviations of the random variable $\diff(N)$ over many independent trials.  The respective data can be found in \tabrefs{IDLA}{rndstack}.
        }
    \label{fig:radiusdiff2}
\end{figure}

An interesting question is whether the runtime could be reduced further by starting from a random odometer approximation~$\widetilde{u}_1$ instead of the deterministic approximation~$u_1$.  One approach is to draw binomials as above (taking $\lambda=0$), and use them to define a ``warped'' Laplacian operator $\widetilde{\Delta}$, given by
	\[ \widetilde{\Delta} f (x) = \sum_{y \sim x} \frac{B_{yx}}{B_y} f(y) - f(x). \]
Here $B_y = u_1(y)$, and~$B_{yx} = R((y,x),u_1(y))$ is the binomial associated to the directed edge $(y,x)$.  We then take~$\widetilde{u}_1$ to be the solution to the variational problem \eqref{variational}, with~$\Delta$ replaced by~$\widetilde{\Delta}$.  This problem can be formulated as a linear program: minimize $\sum_x \widetilde{u}_1(x)$ subject to the constraints~$\widetilde{u}_1 \geq 0$ and $\widetilde{\Delta} \widetilde{u}_1 \leq 1 - N\delta_o$.  One could even iterate this construction, using~$\widetilde{u}_1$ to draw new binomials and get a new warping~$\widetilde{\widetilde{\Delta}}$ and a new approximation~$\widetilde{\widetilde{u}}_1$.  A small number of iterations should suffice to bring the approximation very close to the true odometer.  The main computational issue is how to quickly solve (or even approximately solve) these linear programs, which are sparse but quite large: the number of variables is about $N$.  We achieved some modest speedup with this kind of approach, but not enough to justify the additional complexity.

To measure the circularity of the IDLA cluster, we computed the complex moments
	\[ M_m(A_N) = \sum_{z \in A_N} \left( \frac{z}{r} \right)^m \]
for $m=1,\ldots,100$.  Here $r=\sqrt{N/\pi}$, and we view $z\in A_N$ as a point in the complex plane by identifying $\Z^2$ with $\Z + i\Z$.  These moments obey a central limit theorem~\cite{JLS2}: $M_m(A_N)/\sqrt{N}$ converges in distribution as $N \to \infty$ to a complex Gaussian with variance $1/(m+1)$.
The distribution of the real part of $M_m(A_N)/\sqrt{N}$ is shown in \figref{moments}.

The expected value of the difference $\diff(N)$ between outradius and inradius grows logarithmically in~$N$:
the data in the third column of \tabref{IDLA}, graphed in \figref{radiusdiff2}(a), fit to
\[
    \EE\, \diff(N) = 0.528 \ln(N) - 0.457
\]
with a coefficient of determination of $R^2 = 0.99994$.  Error bars in figure~\figref{radiusdiff2}(a) show standard deviations of the random variable $\diff(N)$.  

Since more than one reader has remarked to us that the straight line fit in \figref{radiusdiff2}(a) looks ``too good to be true,'' we comment briefly on why we believe it comes out this way.  The random variable $\diff(N)$ measures the \emph{largest} fluctuation of $A_N$ from circularity (over all directions).  Very roughly speaking, since we believe the fluctuations in different directions are close to indpendent, $\diff(N)$ behaves like the maximum of many indpendent random variables, which is highly concentrated.  
Note that the size of the standard deviation, represented by the error bars in \figref{radiusdiff2}(a), is approximately constant: it does not grow with $N$.  This finding is consistent with the connection with Gaussian free field revealed in~\cite{JLS2}.  Indeed, if $M_N$ is the maximum of the discrete two-dimensional Gaussian free field in an $N\times N$ box, then the mean $\EE M_N$ has order $\log N$, and the sequence of random variables $\{M_N - \EE M_N\}_{N \geq 1}$ is tight~\cite{BZ}.  Therefore it is natural to believe (although still unproved) that the variance of $M_N$ has order $1$, and that it remains order $1$ if the maximum is taken over the boundary of a discrete ball instead of a box.

\subsection{Low-Discrepancy Random Stack}
\label{sec:exp:rndstack}


\begin{table*}[tb]
\begin{center}
\scalebox{0.71}{
\begin{tabular}{r@{=}lr@{ }lr@{$\pm$}lr@{$\pm$}lr@{$\pm$}lr}
\hline
\hline
\addlinespace[1mm]
\multicolumn{2}{c}{\bf Number of } &
\multicolumn{2}{c}{\bf Average} &
\multicolumn{2}{c}{\bf Radius} &
\multicolumn{2}{c}{\multirow{2}{*}{\bf $\bm{\|u_1-u\|_1/N}$}} &
\multicolumn{2}{c}{\multirow{2}{*}{\bf $\bm{\max |u_1-u|}$}} &
\multicolumn{1}{c}{\bf Number}
\\
\multicolumn{2}{c}{\bf chips $\bm N$} &
\multicolumn{2}{c}{\bf Runtime} &
\multicolumn{2}{c}{\bf Difference} &
\multicolumn{2}{c}{} &
\multicolumn{2}{c}{} &
\multicolumn{1}{c}{\bf of runs}
\\
\addlinespace[1mm]
\hline
\hline
\addlinespace[1mm]

$2^{10}$ &       1,024 & 3.16 & ms    & 1.026 & 0.209 &  1.34 & 0.16 &    6.00 &   0.90 & $5 \cdot 10^{5}$ \\ 
$2^{11}$ &       2,048 & 6.21 & ms    & 1.183 & 0.180 &  1.47 & 0.17 &    6.83 &   0.94 & $5 \cdot 10^{5}$ \\ 
$2^{12}$ &       4,096 & 12.0 & ms    & 1.256 & 0.188 &  1.60 & 0.18 &    7.65 &   1.00 & $5 \cdot 10^{5}$ \\ 
$2^{13}$ &       8,192 & 23.9 & ms    & 1.314 & 0.176 &  1.73 & 0.19 &    8.52 &   1.04 & $5 \cdot 10^{5}$ \\ 
$2^{14}$ &      16,384 & 49.7 & ms    & 1.405 & 0.154 &  1.86 & 0.20 &    9.40 &   1.07 & $5 \cdot 10^{5}$ \\ 
$2^{15}$ &      32,768 & 0.10 & sec   & 1.444 & 0.154 &  1.99 & 0.21 &    10.3 &    1.1 & $5 \cdot 10^{5}$ \\ 
$2^{16}$ &      65,536 & 0.21 & sec   & 1.522 & 0.160 &  2.11 & 0.22 &    11.2 &    1.2 & $5 \cdot 10^{5}$ \\ 
$2^{17}$ &     131,072 & 0.45 & sec   & 1.583 & 0.144 &  2.23 & 0.23 &    12.2 &    1.2 & $5 \cdot 10^{5}$ \\ 
$2^{18}$ &     262,144 & 0.93 & sec   & 1.646 & 0.142 &  2.35 & 0.24 &    13.2 &    1.2 & $5 \cdot 10^{5}$ \\ 
$2^{19}$ &     524,288 & 1.90 & sec   & 1.694 & 0.135 &  2.46 & 0.24 &    14.1 &    1.3 & $5 \cdot 10^{5}$ \\ 
$2^{20}$ &   1,048,576 & 3.88 & sec   & 1.753 & 0.124 &  2.59 & 0.26 &    15.1 &    1.3 & $5 \cdot 10^{5}$ \\ 
$2^{21}$ &   2,097,152 & 7.96 & sec   & 1.808 & 0.124 &  2.73 & 0.28 &    16.2 &    1.4 & $5 \cdot 10^{4}$ \\ 
$2^{22}$ &   4,194,304 & 0.27 & min   & 1.850 & 0.117 &  2.86 & 0.29 &    17.3 &    1.4 & $5 \cdot 10^{4}$ \\ 
$2^{23}$ &   8,388,608 & 0.55 & min   & 1.893 & 0.114 &  2.98 & 0.30 &    18.4 &    1.4 & $5 \cdot 10^{4}$ \\ 
$2^{24}$ &  16,777,216 & 1.13 & min   & 1.942 & 0.109 &  3.11 & 0.31 &    19.4 &    1.5 & $5 \cdot 10^{3}$ \\ 
$2^{25}$ &  33,554,432 & 2.32 & min   & 1.983 & 0.109 &  3.25 & 0.33 &    20.6 &    1.5 & $5 \cdot 10^{3}$ \\ 
$2^{26}$ &  67,108,864 & 4.74 & min   & 2.030 & 0.106 &  3.35 & 0.32 &    21.6 &    1.5 & $5 \cdot 10^{3}$ \\ 
$2^{27}$ & 134,217,728 & 9.72 & min   & 2.070 & 0.093 &  3.51 & 0.35 &    22.9 &    1.5 & $5 \cdot 10^{2}$ \\ 
$2^{28}$ & 268,435,456 & 0.33 & hours & 2.108 & 0.091 &  3.61 & 0.36 &    24.1 &    1.6 & $5 \cdot 10^{2}$ \\

\hline \hline
\end{tabular}}
\end{center}
\caption{Simulation results for low-discrepancy random stack.
    The given runtime is the total runtime of the calculation of one cluster of the given size on one core of an AMD Opteron processor 8222.
    The next column shows the difference between the outradius and inradius of the occupied cluster $A_N$.
    The fourth and fifth columns give two measurements of the error of our odometer approximation~$u_1$,
    the total absolute error and maximum pointwise error.
    The values shown are averages and standard deviations over many independent trials;
    the last column shows the number of trials.
}
\label{tab:rndstack}
\end{table*}

In the rotor-router model (\textsection\ref{sec:exp:RR}), the neighbors are served in a maximally balanced manner, while in IDLA (\textsection\ref{sec:exp:IDLA}), the rotor stack is completely random.
Following a suggestion of James Propp, we examine a model which combines both features by using low-discrepancy random stacks.
In this model the neighbors are served in a similarly balanced manner as in the rotor-router model.  The rotor stacks consist of blocks of length~$4$, chosen independently and uniformly at random from among the~$24$ permutations of~\rotseq{NESW}.
\old{
The rotor stacks consist of blocks
\rotseq{NSEW},
\rotseq{NSWE},
\rotseq{SNEW},
\rotseq{SNWE},
\rotseq{EWNS},
\rotseq{EWSN},
\rotseq{WENS},
\rotseq{WESN} chosen independently and uniformly at random.
}
Hence the rotor stack is random, but still satisfies
$|R(e,n)-R(e',n)|\leq1$ for all $n$ and all edges $e$ and $e'$ such that $\source(e)=\source(e')$.

This model can be implemented with our method in the same way as IDLA.
\figref{radiusdiff2:rndstack} gives averages and standard deviations
for the radius difference $\diff(N)$ up to $N=2^{28}=268,435,456$.
In contrast to IDLA, the difference between inradius and outradius now grows slower than logarithmically in~$N$, and is not much larger than the corresponding difference for the rotor-router model.
In fact, the data points $\EE\, \diff(N)$
of \tabref{rndstack} fit to
\[
\EE\, \diff(N) = 1.018\ln\ln(N) - 0.919
\]
with a coefficient of determination of $R^2 = 0.998$.  Of course, it is very hard to distinguish empirically between slowly growing functions such as $\ln \ln (N)$ and $\sqrt{\ln(N)}$, so we cannot be sure of the exact growth rate; among several functions we tried, $\ln \ln (N)$ had the best fit.
The very slow growth of $\diff(N)$ for the low-discrepancy random stack suggests that the extremely good circularity of the rotor-router model is mainly due to its low discrepancy rather than its deterministic nature.


\section{Further Directions}

We have proved that the abelian stack model can be predicted exactly by correcting an initial approximation.  In experiments, we found that the correction step is quite fast if a good initial approximation is available.  It would be interesting to investigate what other classes of cellular automata can be predicted quickly and exactly by correcting an initial approximation.

The abelian sandpile model in $\Z^2$ produces beautiful examples of pattern formation that remain far from understood \cite{DSC,FLP}.   Using Lemma 2.3 of~\cite{FLP}, it should be possible to characterize the sandpile odometer function in a manner similar to \thmref{odomcharacterization}.  In this characterization, the recurrent sandpile configurations play a role analogous to the acyclic rotor configuration in~\thmref{odomcharacterization}.  The remaining challenge would then be to find a good approximation to the sandpile odometer function.

\section*{Acknowledgments}

We thank James Propp for many enlightening conversations and for comments on several drafts of this article.  We thank David Wilson for suggesting the AES random number generator and for tips about how to use it effectively.

%

\bibliographystyle{abbrvnat}
\bibliography{fast-simulation}

\end{document}